\documentclass[11pt]{amsart}
\usepackage[margin=1in]{geometry}
\usepackage[utf8]{inputenc}
\usepackage[T1]{fontenc}
\usepackage[english]{babel}
\usepackage{lmodern}
\usepackage{graphicx}
\usepackage{graphbox}  
\usepackage{array}
\usepackage{wrapfig}
\usepackage[subrefformat=parens,labelformat=parens]{subfig}
\usepackage{caption}
\usepackage{paralist}
\usepackage{enumitem}  
\usepackage{float}
\usepackage[linesnumbered,ruled,vlined]{algorithm2e}
\usepackage{amsmath}
\usepackage{amsthm}
\usepackage{amssymb}
\usepackage{amsfonts}

\graphicspath{{./figures/}}

\makeatletter
\renewcommand{\p@enumii}{}
\makeatother

\usepackage[usenames]{color}
\definecolor{hellblau}{rgb}{0.2,0.4,1} 
\definecolor{dunkelblau}{rgb}{0,0,0.8}
\definecolor{dunkelgruen}{rgb}{0,0.5,0}

\usepackage[
pdftex,
colorlinks,
linkcolor=dunkelblau,
urlcolor=dunkelblau,
citecolor=dunkelgruen,
bookmarks=true,
linktocpage=true,
pdfauthor={On-Hei Solomon Lo},
pdfsubject={}
]{hyperref}
\usepackage{pdfpages}
\allowdisplaybreaks

\theoremstyle{plain}
\newtheorem{satz}{Satz}  
\newtheorem{theorem}[satz]{Theorem}
\newtheorem{lemma}[satz]{Lemma}

\newtheorem{proposition}[satz]{Proposition}
\newtheorem{corollary}[satz]{Corollary}

\theoremstyle{remark}

\theoremstyle{definition}

\newtheorem{conjecture}[satz]{Conjecture}

\numberwithin{satz}{section}

\begin{document}
	\title[Graphs with no $K_{3,4}$ minor]{A characterization of graphs with no $K_{3,4}$ minor}
	\author[O.-H.S. Lo]{On-Hei Solomon Lo}
	\thanks{This research was partially supported by the Fundamental Research Funds for the Central Universities.}
	\address{School of Mathematical Sciences,
		Key Laboratory of Intelligent Computing and Applications (Ministry of Education), Tongji University, 
		Shanghai 200092, China}
	\email{ohsolomon.lo@gmail.com}
		
	\date{}
	
	\maketitle

\begin{abstract}
A complete structural characterization of graphs with no $K_{3,4}$ minor is obtained, and the following consequences are established. 
Every $4$-connected non-planar graph with at least seven vertices and minimum degree at least five contains both $K_{3,4}$ and $K_6^-$ as minors, thereby proving a conjecture of Kawarabayashi and Maharry in a strengthened form. 
Moreover, every $4$-connected graph with no $K_{3,4}$ minor is hamiltonian-connected, extending a theorem of Thomassen, and admits an embedding on the torus.
\end{abstract}

\sloppy

\section{Introduction}
The theory of graph minors has long occupied a central position in structural graph theory. 
A classical point of departure is the Kuratowski--Wagner theorem~\cite{Kuratowski1930,Wagner1937}, which characterizes planar graphs as those containing neither $K_5$ nor $K_{3,3}$ as a minor. 
This theorem initiated the systematic study of excluded-minor structural characterizations and their consequences.

A decisive advance was achieved by Robertson and Seymour in their Graph Minors series. 
They proved that finite graphs are well-quasi-ordered under the minor relation, thereby extending the Kuratowski--Wagner theorem and settling Wagner’s conjecture that every minor-closed class is determined by a finite set of excluded minors~\cite{Robertson2004}. 
Beyond this fundamental finiteness result, they developed a far-reaching structural theory for graphs excluding a fixed minor. 
Their Minor Structure Theorem~\cite{Robertson2003,Robertson1985} shows that such graphs admit a tree-like decomposition into parts that are almost embeddable in surfaces of bounded genus. 
This yields a powerful global framework for minor-closed classes, but by its nature provides only an approximate structure rather than an exact characterization.

An important direction is therefore to obtain precise structural descriptions for particular excluded minors. 
This problem was explicitly emphasized by Lovász~\cite{Lovasz2006} as one of the fundamental research directions in graph minor theory. 
Exact structural characterizations are known only in a limited number of cases and are typically difficult to obtain.

Among the most prominent open problems are the structural characterizations of graphs with no $K_6$ minor and of graphs with no Petersen minor. 
These questions are closely tied to deep conjectures such as Hadwiger’s conjecture and Tutte’s $4$-flow conjecture. 
At present, a complete description of graphs excluding $K_6$ or the Petersen graph is still beyond reach. 
Nonetheless, the subject has continued to draw considerable interest, and a number of substantial partial results have been obtained; see, for example,~\cite{Kawarabayashi2018,Robertson2019}.

Notwithstanding these difficulties, exact structural characterizations are available for a number of smaller excluded minors. 
Wagner~\cite{Wagner1937,Wagner1937a} provided the classical characterizations for $K_5$ and $K_{3,3}$. 
Maharry~\cite{Maharry2000} obtained a corresponding characterization for the cube, Ding~\cite{Ding2013a} for the octahedron, Maharry and Robertson~\cite{Maharry2016} for the Wagner graph, and Ellingham, Marshall, Ozeki, and Tsuchiya~\cite{Ellingham2016} for $K_{2,4}$.
These structural results have led to a wide range of applications, including verifications of Hadwiger’s conjecture for graphs with no $K_5$ minor~\cite{Wagner1937}, investigations into the flexibility of embeddings~\cite{Maharry2017}, and results on hamiltonicity~\cite{Ellingham2019,Lo2024}.

Graphs excluding $K_{3,t}$ minors play a central role in graph minor theory. 
It is well known that every graph embeddable in a fixed surface excludes $K_{3,t}$ for some $t$. 
Moreover, Robertson and Seymour~\cite{Robertson2026} exhibited a finite family of obstructions, basically obtained by summing copies of $K_5$ and $K_{3,3}$, that characterizes graphs of bounded genus in terms of excluded minors.
These facts motivate a deeper investigation into the structure of graphs with no $K_{3,t}$ minor.

Along this line, Oporowski, Oxley, and Thomas~\cite{Oporowski1993} proved that every sufficiently large $3$-connected graph with no $K_{3,t}$ minor contains a large wheel. 
Böhme, Maharry, and Mohar~\cite{Bohme2002} showed that every sufficiently large $7$-connected graph of bounded tree-width necessarily contains a $K_{3,t}$ minor. 
More recently, Kawarabayashi and Maharry~\cite{Kawarabayashi2012} established that every sufficiently large almost $5$-connected non-planar graph contains both a $K_{3,4}$ minor and a $K_6^-$ minor, and they proposed the following strengthening.

\begin{conjecture}[\cite{Kawarabayashi2012}]\label{con:KM}
	Every $5$-connected non-planar graph with at least seven vertices contains both a $K_{3,4}$ minor and a $K_6^-$ minor.
\end{conjecture}

\begin{figure}[!ht]
	\centering
	\includegraphics[scale=1]{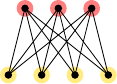}
	\hspace{30pt}
	\includegraphics[scale=1]{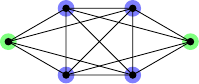}
	\caption{The graphs $K_{3,4}$ (left) and $K_6^-$ (right).}
	\label{fig:K34K6-}
\end{figure}

It has been remarked in~\cite{Chen2012} that a complete characterization of graphs with no $K_{3,t}$ minor ``seems extremely hard to obtain.'' 
Even for the first open case $K_{3,4}$, the problem is described as ``hard'' and ``very challenging'' in~\cite{Kawarabayashi2012,Ding2018}. 
Note that the graph $K_{3,4}$ lies between $K_{3,3}$ and the Petersen graph in the minor order. 

The main objective of this paper is to provide a complete structural characterization of graphs with no $K_{3,4}$ minor. We begin by analyzing the $4$-connected members of this class and then show how all graphs excluding $K_{3,4}$ can be constructed from these fundamental building blocks. Since Maharry and Slilaty~\cite{Maharry2012,Maharry2017a} have already characterized the projective-planar graphs with no $K_{3,4}$ minor, the central focus of our work is on the non-projective-planar case, for which preliminary progress was reported in our earlier paper~\cite{Lo2025}.

The main structural theorem of this paper is as follows. Here, patch graphs, introduced in~\cite{Maharry2012}, are recalled in Section~\ref{sec:pp}, and the family of oloidal graphs is defined in Section~\ref{sec:npp}.

\begin{theorem}\label{thm:4cK34mf}
	A graph $G$ is $4$-connected and has no $K_{3,4}$ minor if and only if one of the following holds:
	\begin{itemize}
		\item $G$ is a $4$-connected planar graph.
		\item $G$ is a $4$-connected non-planar subgraph of a reduced patch graph.
		\item $G$ is isomorphic to $K_6$.
		\item $G$ is an oloidal graph.
	\end{itemize}
\end{theorem}

Several consequences follow from Theorem~\ref{thm:4cK34mf}. 
First, we resolve Conjecture~\ref{con:KM} in a strengthened form.

\begin{theorem}\label{thm:K34K6-}
	Every $4$-connected non-planar graph $G$ with minimum degree at least five contains $K_6^-$ as a minor, and also contains $K_{3,4}$ as a minor if $G$ is not isomorphic to $K_6$.
\end{theorem}

Our second application concerns hamiltonian properties under forbidden minors. 
Tutte~\cite{Tutte1956} proved that every $4$-connected graph with no $K_{3,3}$ minor is hamiltonian, and Thomassen~\cite{Thomassen1983} strengthened this result by showing that such graphs are hamiltonian-connected; that is, any two vertices can serve as the end-vertices of a hamiltonian path.
We extend Thomassen’s theorem by relaxing the excluded minor.

\begin{theorem}\label{thm:hc}
	Every $4$-connected graph with no $K_{3,4}$ minor is hamiltonian-connected.
\end{theorem}

Finally, we derive a consequence regarding the genus of graphs excluding a $K_{3,4}$ minor.

\begin{theorem}\label{thm:T}
	Every $4$-connected graph with no $K_{3,4}$ minor admits an embedding on the torus.
\end{theorem}
\subsection*{Organization of the paper}

The paper is structured as follows. In Section~\ref{sec:definitions}, we fix the basic definitions and notation that will be used throughout. Section~\ref{sec:splitters} introduces Seymour's splitter theorem along with its variants, which form essential tools for our subsequent arguments. In Section~\ref{sec:4K34}, we establish the principal structural result, Theorem~\ref{thm:4cK34mf}, providing a characterization of $4$-connected graphs with no $K_{3,4}$ minor. Building upon this, Section~\ref{sec:K34} extends the characterization to all graphs excluding a $K_{3,4}$ minor, showing how they can be constructed from the $4$-connected building blocks. Finally, Section~\ref{sec:applications} is devoted to the derivation of the consequences of Theorem~\ref{thm:4cK34mf}, specifically Theorems~\ref{thm:K34K6-}, \ref{thm:hc}, and \ref{thm:T}.

\section{Definitions and notation}\label{sec:definitions}

Throughout this paper, we consider only finite simple graphs unless otherwise specified. 

A \emph{$k$-cut} in a connected graph $G$ is a subset $S \subseteq V(G)$ with $|S| = k$ such that $G - S$ is disconnected. The graph $G$ is \emph{$k$-connected} if $|V(G)| > k$ and $G$ has no $k'$-cut for any $k' < k$.
A $3$-connected graph $G$ is \emph{almost $4$-connected} if either $|V(G)| \le 6$, or every $3$-cut of $G$ separates $G$ into a single vertex and a remaining component.
A $3$-connected graph $G$ on at least five vertices is \emph{internally $4$-connected} if either $G$ is isomorphic to $K_{3,3}$, or every $3$-cut of $G$ is an independent set that separates $G$ into a single vertex and a remaining component. In particular, every $4$-connected graph is both almost $4$-connected and internally $4$-connected.

The \emph{degree} of a vertex in a graph $G$ is the number of its neighbors in $G$. A graph $G$ is \emph{$k$-degenerate} if every subgraph of $G$ contains a vertex of degree at most $k$. Equivalently, a graph on $n$ vertices is $k$-degenerate if and only if its vertices can be ordered as $v_1, v_2, \dots, v_n$ such that each vertex $v_i$ has at most $k$ neighbors among $v_1, \dots, v_{i-1}$.

A graph $H$ is a \emph{minor} of a graph $G$ if it can be obtained from $G$ by deleting vertices, deleting edges, and contracting edges. This is denoted by $G \succeq H$, and, when no ambiguity arises, we also say that $G$ contains $H$ as a minor, or $G$ has a minor of $H$.

It is well known that $G \succeq H$ if and only if there exists a mapping \( \mu \) that assigns to each vertex of \( H \) a pairwise disjoint subset of \( V(G) \) and to each edge of \( H \) a path in \( G \), such that for every \( v \in V(H) \), the induced subgraph \( G[\mu(v)] \) is connected, the paths corresponding to distinct edges are pairwise internally disjoint, and for each \( uv \in E(H) \), the path \( \mu(uv) \) joins a vertex in \( \mu(u) \) to a vertex in \( \mu(v) \) with no internal vertex contained in \( \bigcup_{w \in V(H)} \mu(w) \). We call such a mapping \( \mu \) a \emph{model} of $H$, witnessing that \( H \) is a minor of \( G \). Moreover, if \( G \) is connected, we may assume that $\bigcup_{w \in V(H)} \mu(w) = V(G)$; in this case, we call \( \mu \) a \emph{spanning model}.

Assume that $G$ contains a model $\mu$ of $H$. If a vertex $v \in V(H)$ has degree at most four, we may assume that the subgraph $G[\mu(v)]$ contains a spanning path $P$ such that, for every edge $uv \in E(H)$, the intersection of $\mu(v)$ with the path $\mu(uv)$ is an end-vertex of $P$, and each end-vertex of $P$ is incident with at least two paths corresponding to edges incident with $v$. If $\mu(v)$ consists of a single vertex for every $v \in V(H)$, then we say that $G$ contains a \emph{subdivision} of $H$. Note that if every vertex of $H$ has degree three, then the existence of an $H$-minor is equivalent to the existence of a subdivision of $H$.

Let $A$ be a set of edges not contained in $G$. We write $G + A$ for the graph obtained from $G$ by adding the edges in $A$. Let $D$ be a set of vertices or edges contained in $G$, and write $G - D$ for the graph obtained from $G$ by deleting the elements of $D$. When $A=\{a\}$ (respectively, $D=\{d\}$), we may write $G+a$ (respectively, $G-d$) for brevity.

For any path $P$ and vertices $u,v \in V(P)$, we denote by $P[u,v]$ the subpath of $P$ with end-vertices $u$ and $v$. A vertex of $P$ is \emph{internal} if it is not an end-vertex, and we write $P(u,v)$ for the (possibly empty) path obtained from $P[u,v]$ by deleting $u$ and $v$.

The \emph{length} of a walk, path, or cycle is the number of edge occurrences it contains. A cycle in a graph embedded on a surface that bounds a face is called a \emph{facial cycle}. More generally, the closed walk that traces the boundary of a face is called a \emph{facial walk}. The \emph{length} of a face is the length of its facial walk. Two vertices of an embedded graph are \emph{cofacial} if they are incident with a common face.

Two edges are \emph{independent} if they have no common end-vertex.

For any positive integer $k$, we write $[k]$ for the set of positive integers at most $k$.

\section{Splitter theorems}\label{sec:splitters}

In many situations the graph $G$ is not given explicitly; instead, we are given a known graph $H \neq G$ such that $G \succeq H$, and we wish to exploit this information to understand the structure of $G$. By definition, $H$ is obtained from $G$ by a non-trivial sequence of deletions and contractions. Equivalently, one may view $G$ as arising from $H$ by a sequence of minor extensions, each producing a larger graph that still occurs as a minor of $G$. In this way, starting from the fixed minor $H$, we obtain a family of graphs $H'$ with $H \preceq H' \preceq G$. This principle underpins splitter theorems and their variants, which \emph{fix} $H$ by growing it toward $G$.

Such tools have proved highly useful in the literature. The classical result, commonly known as Seymour's splitter theorem~\cite{Seymour1980,Negami1982}, fixes the size of $H$ when both $G$ and $H$ are $3$-connected. Hegde and Thomas~\cite{Hegde2018} introduced a variant that fixes the non-embeddability of $H$. We mention that Ding and Iverson~\cite{Ding2014} developed a connectivity-fixing tool that grows $H$ to certain \emph{twists}, which do not necessarily contain $H$ as a minor; they used it to characterize the minor-minimal internally $4$-connected non-projective-planar graphs. We refer to~\cite{Thomas1999,Norin2015} for a comprehensive exposition.

\subsection{Seymour's splitter theorem}

Let $H$ be a graph and let $v \in V(H)$. Construct a graph $H'$ from $H$ as follows: delete $v$, introduce two adjacent vertices $v_1$ and $v_2$, and distribute the neighbors of $v$ between $v_1$ and $v_2$ so that each former neighbor of $v$ is adjacent to exactly one of $v_1$ and $v_2$, and each of $v_1$ and $v_2$ receives at least two such neighbors. In particular, this requires that $v$ has degree at least four. We say that $H'$ is obtained from $H$ by \emph{splitting} the vertex $v$. The reverse operation consists of contracting an edge whose end-vertices have no common neighbor.

The theorem below is a graph-theoretic consequence of Seymour’s splitter theorem, and was also obtained independently by Negami. Recall that a \emph{wheel} is the graph formed from a cycle by adding a single vertex adjacent to every vertex of the cycle.

\begin{theorem}[\cite{Seymour1980,Negami1982}]\label{thm:Seymour}
	Let $G$ and $H$ be $3$-connected graphs. Suppose that $G$ contains $H$ as a minor and, moreover, that if $H$ is a wheel, then it is a largest wheel minor of $G$. Then $G$ can be obtained from $H$ by a sequence of operations, each of which either adds an edge between two non-adjacent vertices or splits a vertex.
\end{theorem}

\subsection{Internally $4$-connected graphs}

While Seymour's splitter theorem applies to $3$-connected graphs, Johnson and Thomas~\cite{Johnson2002} asked for an analogue tailored to internally $4$-connected graphs. They announced such a theorem in~\cite[Section~3]{Johnson2002} (see also~\cite{Thomas1999}); however, a proof has not yet appeared in the literature. Here we collect several of their structural results that will be used later.

Recall that if $G$ contains $H$ as a minor but contains no graph obtained from $H$ by splitting a vertex as a minor, then $G$ contains a subdivision of $H$.

Now let $G$ and $H$ be graphs such that $G$ contains a subdivision of $H$. Then there exists a mapping $\eta$ that maps the vertices of $H$ to distinct vertices of $G$ and maps each edge of $H$ to a path in $G$, such that these paths are pairwise internally disjoint and, for every $uv \in E(H)$, the path $\eta(uv)$ has end-vertices $\eta(u)$ and $\eta(v)$ and no internal vertex of $\eta(uv)$ lies in $\{\eta(v) : v \in V(H)\}$.

Let $\eta(H)$ denote the subgraph of $G$ formed by the image of $\eta$; we also say that $\eta(H)$ is a subdivision of $H$ in $G$. Define $\eta(V(H)) := \{\eta(v) : v \in V(H)\}$. A vertex is called \emph{internal} if it lies in $V(G) \setminus \eta(V(H))$. The paths $\eta(e)$ for $e \in E(H)$ are called the \emph{segments} of $\eta$.

For $v \in V(H)$, the \emph{domain} of $v$ with respect to $\eta$ is the set consisting of $\eta(v)$ together with all internal vertices of the segments having $\eta(v)$ as an end-vertex. The \emph{closed domain} of $v$ with respect to $\eta$ is the set of all vertices on the segments having $\eta(v)$ as an end-vertex.

Johnson and Thomas~\cite{Johnson2002} introduced a special type of subdivision, called a \emph{lexicographically maximal subdivision}, which we refer to as a \emph{JT-subdivision}. For the technical definition of a JT-subdivision, we refer the reader to \cite{Johnson2002}. We mention two facts about JT-subdivisions, although the first is not needed in the proofs of this paper. 

\begin{proposition}[\cite{Johnson2002}] \label{pro:JT}
	Let $G$ and $H$ be graphs. Then the following statements hold:
	\begin{itemize}
		\item If \( G \) contains a subdivision of \( H \), then \( G \) also contains a JT-subdivision of \( H \). 
		\item If \( G \) contains \( H \) as a spanning subgraph, then this subgraph, when viewed as a subdivision of \( H \), is a JT-subdivision of \( H \).
	\end{itemize}
\end{proposition}

The following lemma collects two fundamental properties of JT-subdivisions, as established in~\cite[(5.2)]{Johnson2002} and~\cite[(7.2)]{Johnson2002}, and will be used repeatedly in Section~\ref{sec:npp}.

\begin{lemma}[\cite{Johnson2002}] \label{lem:JT}
	Let \( G \) and \( H \) be internally \(4\)-connected graphs such that \( G \) contains a JT-subdivision \( \eta(H) \) of \( H \). Then the following statements hold:
	\begin{itemize}
		\item If \( G \) does not contain, as a minor, any graph obtained from \( H \) by splitting a vertex, then every segment of \( \eta(H) \) is an induced path in \( G \).
		\item Let \( u \) be a vertex of \( H \) of degree three, and let \( e_1 \) and \( e_2 \) be two distinct edges of \( H \) incident with \( u \). If there exists an edge of \( G \) joining a vertex \( v \) in \( \eta(e_1) - \eta(u) \) to an internal vertex of \( \eta(e_2) \), then $v$ is adjacent to $\eta(u)$ in \( \eta(e_1) \).
	\end{itemize}
\end{lemma}

\subsection{Non-embeddable extensions}

Suppose that $G \succeq H$, where $H$ is embeddable on a surface but $G$ is not embeddable on that surface. In this situation, $G \neq H$, and we may exploit this to \emph{fix} the non-embeddability of $G$ relative to $H$. Hegde and Thomas~\cite{Hegde2018} developed a powerful framework for addressing such cases. For our purposes, we only need a restricted version: we assume that $G$ is non-planar, $H$ is planar, and both graphs are $4$-connected.

Let $H$ be a $4$-connected planar graph, and let $H'$ be obtained from $H$ by splitting a vertex $v \in V(H)$ into two vertices $v_1,v_2 \in V(H')$. We say that $H'$ is obtained by a \emph{planar split} of $v$ if $H'$ is planar, and by a \emph{non-planar split} otherwise. Equivalently, a split is planar if and only if the neighbors assigned to each of $v_1$ and $v_2$ form contiguous intervals in the cyclic order of the neighbors of $v$ in the planar embedding of $H$. 
After a planar split, the facial cycles of $H'$ correspond naturally to those of $H$, with all lengths preserved except for exactly two facial cycles $C_1$ and $C_2$ of $H$, which correspond to facial cycles in $H'$ whose lengths increase by one; these two cycles in $H'$ share precisely the edge $v_1v_2$. In this case, the planar split is said to be \emph{along $C_1$} (and \emph{along $C_2$}, respectively).

The following result is a weakened form of~\cite[Theorem~1.2]{Hegde2018} and plays a key role in the proof of Theorem~\ref{thm:K34K6-} in Section~\ref{sec:K34K6-}.

\begin{theorem}[\cite{Hegde2018}]\label{thm:HT}
	Let $G$ and $H$ be $4$-connected graphs. Suppose that $G$ contains $H$ as a minor, that $G$ is non-planar, and that $H$ is planar. Then $G$ contains a graph $H'$ as a minor such that one of the following holds:
	\begin{itemize}
		\item $H'$ is obtained from $H$ by joining two non-cofacial vertices.
		
		\item There exist distinct vertices $v_1,v_2,v_3,v_4$ appearing in this order on a facial cycle of $H$ such that $H'$ is obtained from $H$ by joining $v_1$ with $v_3$ and joining $v_2$ with $v_4$.
		
		\item $H'$ is obtained from $H$ by performing a non-planar split.
		
		\item There exist non-adjacent cofacial vertices $u,v$ of $H$ such that $H'$ is obtained by first performing a planar split of $v$ into $v_1,v_2$ so that, in the intermediate graph, $u$ and $v_2$ are non-cofacial, and then joining $u$ with $v_2$.
		
		\item There exist facial cycles $C_1$ and $C_2$ of $H$ sharing a common edge $uv$ such that $H'$ is obtained by performing a planar split of $u$ along $C_1$ into $u_1,u_2$ and a planar split of $v$ along $C_2$ into $v_1,v_2$, with $u_1$ adjacent to $v_1$ in the intermediate graph, and then joining $u_2$ with $v_2$.
		
		\item There exist distinct vertices $u,v,w$ on a facial cycle $C$ of $H$, where $u$ is adjacent to neither $v$ nor $w$, such that $H'$ is obtained by first performing a planar split of $u$ along $C$ into $u_1,u_2$ so that $u_1,u_2,v,w$ appear on the new facial cycle of the intermediate graph, and then joining $u_1$ with $v$ and joining $u_2$ with $w$.
		
		\item There exist non-adjacent vertices $u,v$ on a facial cycle $C$ of $H$ such that $H'$ is obtained by performing a planar split of $u$ into $u_1,u_2$ and a planar split of $v$ into $v_1,v_2$, both along $C$, so that $u_1,u_2,v_1,v_2$ appear on the new facial cycle of the intermediate graph, and then joining $u_1$ with $v_1$ and joining $u_2$ with $v_2$.
	\end{itemize}
\end{theorem}

\section{$4$-connected graphs with no $K_{3,4}$ minor}\label{sec:4K34}

In this section, we establish our principal result, Theorem~\ref{thm:4cK34mf}. We begin by presenting the characterization of projective-planar graphs due to Maharry and Slilaty in Section~\ref{sec:pp}, and then proceed to the non-projective-planar case in Section~\ref{sec:npp}, which constitutes the most technical part of our analysis. The proof of Theorem~\ref{thm:4cK34mf} is then completed in Section~\ref{sec:4cK34mf}.

\subsection{The projective-planar case}\label{sec:pp}

Maharry and Slilaty established two characterizations of graphs embedded in the projective plane with no $K_{3,4}$ minor: one in terms of \emph{patch graphs}~\cite{Maharry2012} and the other in terms of \emph{M\"obius hyperladders}~\cite{Maharry2017a}. In this section we recall the former. We review the construction of patch graphs and state the main result of~\cite{Maharry2012}, Theorem~\ref{thm:MS}. Moreover, we prove that every patch graph is $4$-degenerate.

A \emph{construct} is a pair $(G,\mathcal{P})$ consisting of a multigraph $G$ embedded in the projective plane and a family $\mathcal{P}$ of faces of length $4$ in that embedding, which are called \emph{patches} of $G$. A construct $(G,\mathcal{P})$ can be transformed into another construct $(G',\mathcal{P}')$ by applying one of three operations, called \emph{$H$-}, \emph{$Y$-}, and \emph{$I$-patchings}. The graphs depicted in Figure~\ref{fig:HYI} are called the \emph{$H$-}, \emph{$Y$-}, and \emph{$I$-pieces}, respectively; each is bounded by a cycle of length $4$.

\begin{figure}[!ht]
	\centering
	\includegraphics[scale=1]{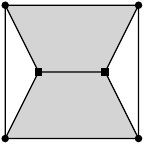}
	\hspace{30pt}
	\includegraphics[scale=1]{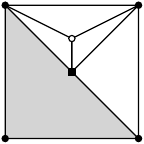}
	\hspace{30pt}
	\includegraphics[scale=1]{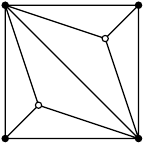}
	\caption{The $H$-piece (left), $Y$-piece (middle), and $I$-piece (right). Applying the corresponding patching adds two, one, or no new patches, respectively, indicated by the shaded regions.}
	\label{fig:HYI}
\end{figure}

To apply an $H$-patching to $(G,\mathcal{P})$, choose a patch $P \in \mathcal{P}$ and form $G'$ by identifying the facial walk of $P$ with the boundary cycle of the $H$-piece, allowing the latter to be flipped or rotated prior to the identification. We say that the $H$-piece is \emph{patched to} $P$. The new family $\mathcal{P}'$ is obtained from $\mathcal{P}$ by removing $P$ and adding the two new patches contained in the $H$-piece. The definitions of $Y$- and $I$-patchings are analogous.

The \emph{initial patch construct} $(G_0,\mathcal{P}_0)$ is the construct in which $G_0$, called the \emph{initial patch graph}, consists of two vertices joined by four parallel edges and is embedded in the projective plane as depicted in Figure~\ref{fig:icon}, and $\mathcal{P}_0$ consists of the unique face of length $4$. A \emph{patch construct} is a construct obtained from the initial patch construct by a sequence of $H$-, $Y$-, and $I$-patchings. A \emph{patch graph} is the multigraph $G$ arising from a patch construct $(G,\mathcal{P})$.

\begin{figure}[!ht]
	\centering
	\includegraphics[scale=1]{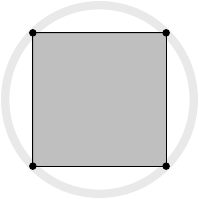}
	\caption{The initial patch construct $(G_0,\mathcal{P}_0)$ embedded in the projective plane. The outer circle represents the cross-cap, and the shaded region is the unique face of length $4$.}
	\label{fig:icon}
\end{figure}

Maharry and Slilaty~\cite{Maharry2012} proved the following result.\footnote{In the original setting of~\cite{Maharry2012}, an additional patching operation is considered. This operation can be simulated by first applying a $Y$-patching and then an $I$-patching, followed by deleting an appropriate degree three vertex from the $I$-piece. Consequently, the main theorem of~\cite{Maharry2012} remains valid under our simplified framework.}

\begin{theorem}[\cite{Maharry2012}]\label{thm:MS}
	Any almost $4$-connected, non-planar, projective-planar graph with no $K_{3,4}$ minor is a subgraph of a patch graph or is isomorphic to $K_6$. Moreover, every patch graph contains no $K_{3,4}$ minor.
\end{theorem}

Let the pieces obtained from the $H$-, $Y$-, and $I$-pieces by removing the white vertices indicated in Figure~\ref{fig:HYI} be called the \emph{reduced $H$-}, \emph{$Y$-}, and \emph{$I$-pieces}. Thus, the reduced $H$-piece coincides with the $H$-piece, while the reduced $Y$- and $I$-pieces are obtained by deleting one and two vertices, respectively. A \emph{reduced patch graph} is defined analogously to a patch graph, except that only \emph{reduced pieces} are used in the patching operations. Equivalently, a reduced patch graph can be obtained from a patch graph by removing all white vertices from the pieces used in the patching operations. It is easy to see that any $4$-connected subgraph of a patch graph is also a subgraph of a reduced patch graph.

While every projective-planar graph is $5$-degenerate, we establish that patch graphs have strictly smaller degeneracy. This property will be used in the proof of Theorem~\ref{thm:K34K6-} in Section~\ref{sec:K34K6-}.

\begin{lemma}\label{lem:pg4d}
	Every patch graph is $4$-degenerate.
\end{lemma}

\begin{proof}
	Observe that in a patch graph, the white vertices originating from a $Y$- or $I$-piece always have degree at most $3$. Therefore, it suffices to show that every reduced patch graph is $4$-degenerate.
	
	Let $G$ be a reduced patch graph. There exists a sequence of reduced patch graphs $G_0, G_1, \dots, G_t$, where $G_0$ is the initial patch graph, $G_t = G$, and for each $i \in [t]$, the graph $G_i$ is obtained from $G_{i-1}$ by patching a reduced $R_i$-piece with $R_i \in \{H,Y,I\}$ to a patch $P_{i-1}$ of $G_{i-1}$. Moreover, we may assume the following:
	\begin{itemize}
\item If $R_i = H$ and, in the reduced $H$-piece used for the patching, one patch $P$ is patched with a reduced $I$-piece while the other patch is patched with a reduced $H$- or $Y$-piece, then $P = P_i$.
		\item If $R_i = Y$ and the patch in the reduced $Y$-piece is patched with a reduced $I$-piece, then $P = P_i$.
	\end{itemize}
	
	We prove the assertion by induction on $t$. Clearly, $G_0$ is $4$-degenerate. Assume $t \ge 1$ and that the assertion holds for all smaller values of $t$.

	First assume that $R_t = H$. By deleting the two square vertices of the reduced $H$-piece (as indicated in Figure~\ref{fig:HYI}), we obtain $G_{t-1}$. By the induction hypothesis, $G_{t-1}$ is $4$-degenerate. Since the deleted vertices have degree at most $3$ in $G_t$, it follows that $G_t$ is also $4$-degenerate.
	
	Next assume that $R_t = Y$. Deleting the square vertex of the reduced $Y$-piece yields $G_{t-1}$. Again, by the induction hypothesis, $G_t$ is $4$-degenerate.
	
	Now assume that $R_t = I$. If $t = 1$, then $G_t$ is immediately seen to be $4$-degenerate. We therefore assume $t > 1$, and hence $P_{t-1}$ does not belong to the patch family of the initial patch construction.
	
	Suppose that $P_{t-1}$ comes from a reduced $H$-piece, and let $P$ denote the other patch of this piece. By our assumptions on the sequence, no reduced $H$- or $Y$-piece is patched to $P$. We delete the square vertices of this reduced $H$-piece so that the first deleted vertex has degree at most $4$ and the second has degree at most $5$ in $G_t$, thereby obtaining a smaller reduced patch graph. The induction hypothesis then implies that $G_t$ is $4$-degenerate.
	
	Finally, suppose that $P_{t-1}$ comes from a reduced $Y$-piece. By our assumptions on the reduced patch graph sequence, this reduced $Y$-piece was patched to $P_{t-2}$ during the reduced $R_{t-1}$-patching. Deleting the square vertex of this reduced $Y$-piece, which has degree at most $4$ in $G_t$, yields either the reduced patch graph $G_{t-2}$ or that obtained from $G_{t-2}$ by patching a reduced $I$-piece to $P_{t-2}$. In either case, the induction hypothesis applies and implies that $G_t$ is $4$-degenerate.
	\end{proof}

\subsection{The non-projective-planar case}\label{sec:npp}

Since projective-planar graphs with no $K_{3,4}$ minor have been characterized by Maharry and Slilaty, we extend the study to the non-projective-planar case. After introducing a class of graphs called \emph{oloidal graphs}, we show that they characterize the $4$-connected non-projective-planar graphs with no $K_{3,4}$ minor.

\subsubsection{Three Archdeacon graphs and oloidal graphs}

We consider the graphs $D_{17}$, $E_{20}$, and $F_4$, shown in Figure~\ref{fig:DEF} with the vertex labels as indicated. Each of these graphs is internally $4$-connected and appears in Archdeacon’s list of \emph{minor-minimal} non-projective-planar graphs~\cite{Archdeacon1981}; the graph names originate from the notation introduced in~\cite{Glover1979}.

\begin{figure}[!ht]
	\centering
	\includegraphics[scale=1]{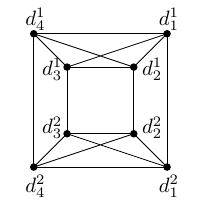}
	\hspace{30pt}
	\includegraphics[scale=1]{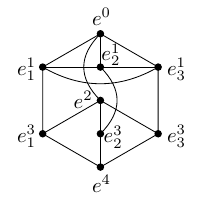}
	\hspace{30pt}
	\includegraphics[scale=1]{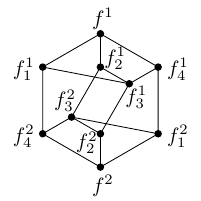}
	\caption{The internally $4$-connected non-projective-planar graphs $D_{17}$ (left), $E_{20}$ (middle), and $F_4$ (right).}
	\label{fig:DEF}
\end{figure}

The symmetries of these graphs facilitate the reduction of the case analysis.

The graph $D_{17}$ has 48 automorphisms, each either fixing $\{d^1_1, d^1_2, d^1_3, d^1_4\}$ or mapping it to $\{d^2_1, d^2_2, d^2_3, d^2_4\}$. Those automorphisms that fix $\{d^1_1, d^1_2, d^1_3, d^1_4\}$ correspond bijectively to the permutations of these four vertices.

The graph $E_{20}$ has 6 automorphisms, corresponding bijectively to the permutations of $\{e^1_1, e^1_2, e^1_3\}$.

The graph $F_4$ has 4 automorphisms, each mapping $f^1_1$ to one of $f^1_1$, $f^1_4$, $f^2_1$, or $f^2_4$.

The following proposition summarizes the results of~\cite[Sections~4 and~6]{Lo2025}.

\begin{proposition}[\cite{Lo2025}]\label{pro:DEF}
	Let \( G \) be a \( 4 \)-connected non-projective-planar graph with no \( K_{3,4} \) minor. Then the following hold:
	\begin{itemize}
		\item \( G \) has a minor of \( D_{17} \), \( E_{20} \), or \( F_4 \).
		\item If \( G \) has a minor of \( D_{17} \) but no minor of \( E_{20} \), then \( G \) contains a spanning subgraph isomorphic to \( D_{17} \).
		\item If \( G \) has a minor of \( E_{20} \) but no minor of \( F_4 \), then \( G \) contains a spanning JT-subdivision of \( E_{20} \), and does not contain as a minor any graph obtained from \( E_{20} \) by splitting a vertex.
		\item If \( G \) has a minor of \( F_4 \), then \( G \) contains a spanning JT-subdivision of \( F_4 \), and does not contain as a minor any graph obtained from \( F_4 \) by splitting a vertex.
	\end{itemize}
\end{proposition}

We define three classes of graphs that play a central role in our characterization of $4$-connected non-projective-planar graphs with no $K_{3,4}$ minor. For $s^1,s^2,s \ge 0$, let $\mathfrak{D}_{s^1,s^2}$, $\mathfrak{E}_s$, and $\mathfrak{F}$ be the graphs depicted in Figure~\ref{fig:DEF4c}. 

The graph $\mathfrak{D}_{s^1,s^2}-\{\delta^1_1,\delta^1_3,\delta^2_2,\delta^2_4\}$ consists of the disjoint union of a path on $s^1+2$ vertices with end-vertices $\delta^1_2$ and $\delta^1_4$, and a path on $s^2+2$ vertices with end-vertices $\delta^2_1$ and $\delta^2_3$. Moreover, every vertex on the former (respectively, latter) path is adjacent to both $\delta^1_1$ and $\delta^1_3$ (respectively, $\delta^2_2$ and $\delta^2_4$). Similarly, $\mathfrak{E}_s-\{\varepsilon^1_1,\varepsilon^1_3,\varepsilon^2,\varepsilon^3_1,\varepsilon^3_2,\varepsilon^3_3,\varepsilon^4\}$ is a path on $s+2$ vertices with end-vertices $\varepsilon^0$ and $\varepsilon^1_2$, where each vertex on the path is adjacent to both $\varepsilon^1_1$ and $\varepsilon^1_3$. 

These paths are referred to as the \emph{spines} of $\mathfrak{D}_{s^1,s^2}$ and the \emph{spine} of $\mathfrak{E}_s$, respectively. We denote the spines of $\mathfrak{D}_{s^1,s^2}$ by $\sigma^1_0\sigma^1_1\cdots\sigma^1_{s^1}\sigma^1_{s^1+1}$ and $\sigma^2_0\sigma^2_1\cdots\sigma^2_{s^2}\sigma^2_{s^2+1}$, where $\delta^1_2=\sigma^1_0$, $\delta^1_4=\sigma^1_{s^1+1}$, $\delta^2_1=\sigma^2_0$, and $\delta^2_3=\sigma^2_{s^2+1}$. The spine of $\mathfrak{E}_s$ is denoted by $\sigma_0\sigma_1\cdots\sigma_s\sigma_{s+1}$, where $\varepsilon^0=\sigma_0$ and $\varepsilon^1_2=\sigma_{s+1}$. A spine is \emph{trivial} if it has length one.

\begin{figure}[!ht]
	\centering
	\includegraphics[scale=1]{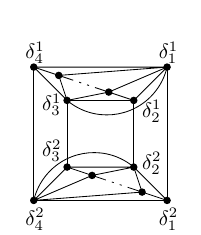}
	\hspace{30pt}
	\includegraphics[scale=1]{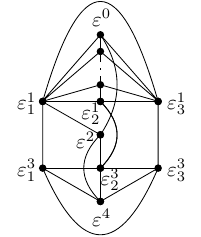}
	\hspace{30pt}
	\includegraphics[scale=1]{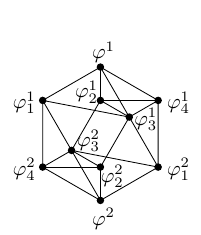}
	\caption{The $4$-connected non-projective-planar graphs $\mathfrak{D}_{s^1,s^2}$ (left), $\mathfrak{E}_{s}$ (middle), and $\mathfrak{F}$ (right).}
	\label{fig:DEF4c}
\end{figure}

We show that the graphs defined above are $4$-connected, non-projective-planar, and contain no $K_{3,4}$ minor.

\begin{lemma}\label{lem:VZ}
	Let $G$ be a graph in which $v_1$ and $v_2$ are adjacent vertices of degree at least four. Let $H$ be obtained from $G$ by contracting the edge $v_1v_2$. If $H$ is $4$-connected, then $G$ is also $4$-connected.
\end{lemma}
\begin{proof}
Suppose for a contradiction that $G$ has a $k$-cut $S$ with $k<4$. Denote by $v \in V(H)$ the vertex obtained by contracting $v_1v_2$. Define $S':=S$ if $\{v_1,v_2\}\cap S=\emptyset$, and $S':=(S\setminus\{v_1,v_2\})\cup\{v\}$ otherwise. Since each of $v_1$ and $v_2$ has degree at least four, no component of $G-S$ consists of precisely one of $v_1, v_2$. It follows that $S'$ is a $(k-1)$- or $k$-cut of $H$, a contradiction.
\end{proof}

\begin{lemma}\label{lem:DEF4con}
	For any $s^1, s^2 \ge 0$, the graph $\mathfrak{D}_{s^1,s^2}$ is $4$-connected and non-projective-planar. For any $s \ge 0$, the graph $\mathfrak{E}_s$ is $4$-connected and non-projective-planar. The graph $\mathfrak{F}$ is $4$-connected and non-projective-planar.
\end{lemma}

\begin{proof}
	By successively contracting edges of the spines (respectively, the spine), Lemma~\ref{lem:VZ} implies that $\mathfrak{D}_{s^1,s^2}$ (respectively, $\mathfrak{E}_s$) is $4$-connected whenever $\mathfrak{D}_{0,0}$ (respectively, $\mathfrak{E}_0$) is $4$-connected. It therefore remains to verify that each of $\mathfrak{D}_{0,0}$, $\mathfrak{E}_0$, and $\mathfrak{F}$ is $4$-connected, which can be done straightforwardly.
	
The non-projective-planarity of the graphs follows from the facts that $\mathfrak{D}_{s^1,s^2}$ and $\mathfrak{E}_s$ contain $\mathfrak{D}_{0,0}$ and $\mathfrak{E}_0$ as minors, respectively, and that $\mathfrak{D}_{0,0}$, $\mathfrak{E}_0$, and $\mathfrak{F}$ contain the non-projective-planar graphs $D_{17}$, $E_{20}$, and $F_4$ as spanning subgraphs, respectively.
\end{proof}

\begin{lemma}\label{lem:to2}
	Let $G$ be a graph with distinct vertices $v_1, v_2, w_0, w_1, \dots, w_s, w_{s+1}$, where $s>2$, such that $P := w_0 w_1 \dots w_s w_{s+1}$ is an induced path, every vertex of $P$ has degree four and is adjacent to both $v_1$ and $v_2$, and $v_1$ and $v_2$ are adjacent. Let $H$ be the graph obtained from $G$ by contracting $w_0 w_1$. If $G$ contains $K_{3,4}$ as a minor, then $H$ also contains $K_{3,4}$ as a minor.
\end{lemma}

\begin{proof}
	Suppose, for a contradiction, that $H$ has no $K_{3,4}$ minor. 
	
	Denote the vertex set of $K_{3,4}$ by $X \cup Y$, where $|X| = 3$ and $|Y| = 4$, such that two vertices of $K_{3,4}$ are adjacent if and only if they belong to different partite sets $X$ and $Y$.
	
	Clearly, we may assume that $G$ is connected. Let $\mu$ be a spanning model of $K_{3,4}$ in $G$.
	
	Note that contracting any edge of $P$ in $G$ yields a graph isomorphic to $H$. For any $i \in \{0,1,\dots,s\}$ and any $v \in V(K_{3,4})$, we have $\{w_i, w_{i+1}\} \not\subseteq \mu(v)$; otherwise, the graph $H$, obtained from $G$ by contracting the edge $w_i w_{i+1}$, would contain a $K_{3,4}$ minor.
	
	For any $i \in [s]$ and $j \in [2]$, if there exists $v \in V(K_{3,4})$ such that $\{w_i, v_j\} \subseteq \mu(v)$, then, since every neighbor of $w_i$ other than $v_j$ is also a neighbor of $v_j$, the graph obtained from $G$ by deleting $w_i$ would contain a $K_{3,4}$ minor. This is a contradiction, because the graph obtained from $G$ by removing $w_i$ and joining $w_{i-1}$ and $w_{i+1}$ is isomorphic to $H$.
	
	Therefore, for each $i \in [s]$, there exists $v \in V(K_{3,4})$ such that $\mu(v) = \{w_i\}$. We claim that every such vertex $v$ belongs to $Y$. Without loss of generality, assume that $i < s$. If the claim fails, then $v \in X$ has degree four in $K_{3,4}$. Since $w_i$ has degree four in $G$, it follows that
	$\{v_1, v_2, w_{i-1}, w_{i+1}\} \subseteq \bigcup_{u \in Y} \mu(u)$.
	Consequently, there exists $w \in Y$ such that $\mu(w) = \{w_{i+1}\}$. Then $w$ is adjacent to at most two vertices of $X$, which is impossible.
	
	From the above discussion, we conclude that for each $i \in [s]$, there exists $y_i \in Y$ such that $\mu(y_i) = \{w_i\}$. Since $s > 2$, the vertex $y_2 \in Y$ is adjacent to at most two vertices of $X$, a contradiction.
\end{proof}

\begin{lemma}\label{lem:DEFK34mf}
	For any $s^1, s^2 \ge 0$, the graph $\mathfrak{D}_{s^1,s^2} + \{\delta^1_1 \delta^2_2,\delta^2_2 \delta^1_3,\delta^1_3 \delta^2_4,\delta^2_4 \delta^1_1\}$ does not contain $K_{3,4}$ as a minor. For any $s \ge 0$, the graph $\mathfrak{E}_s +\{\varepsilon^1_1 \varepsilon^3_2,\varepsilon^1_3 \varepsilon^3_2,\varepsilon^1_3 \varepsilon^2\}$ does not contain $K_{3,4}$ as a minor. The graph $\mathfrak{F}+\{\varphi^1_3 \varphi^2_3\}$ does not contain $K_{3,4}$ as a minor.
\end{lemma}

\begin{proof}
	Suppose, to the contrary, that for some $s^1, s^2 \ge 0$, the graph $\mathfrak{D}_{s^1,s^2} + \{\delta^1_1 \delta^2_2,\delta^2_2 \delta^1_3,\delta^1_3 \delta^2_4,\delta^2_4 \delta^1_1\}$ contains a $K_{3,4}$ minor. Choose such a graph with $s^1+s^2$ minimum. One can directly verify (preferably with a computer) that $\mathfrak{D}_{2,2} + \{\delta^1_1 \delta^2_2,\delta^2_2 \delta^1_3,\delta^1_3 \delta^2_4,\delta^2_4 \delta^1_1\}$ contains no $K_{3,4}$ minor. At least one of $s^1$ and $s^2$ is greater than $2$; otherwise, $\mathfrak{D}_{s^1,s^2} + \{\delta^1_1 \delta^2_2, \delta^2_2 \delta^1_3, \delta^1_3 \delta^2_4, \delta^2_4 \delta^1_1\}$ would be a minor of $\mathfrak{D}_{2,2} + \{\delta^1_1 \delta^2_2, \delta^2_2 \delta^1_3, \delta^1_3 \delta^2_4, \delta^2_4 \delta^1_1\}$, which contains no $K_{3,4}$ minor. Without loss of generality, assume $s^1>2$. Every vertex on the spine with end-vertices $\delta^1_2$ and $\delta^1_4$ has degree four and is adjacent to both $\delta^1_1$ and $\delta^1_3$, and moreover $\delta^1_1$ and $\delta^1_3$ are adjacent. Therefore, contracting the edge $\sigma^1_{s^1}\sigma^1_{s^1+1}$ produces a graph that contains a $K_{3,4}$ minor by Lemma~\ref{lem:to2}. This contradicts the minimality of $s^1+s^2$, since the resulting graph is isomorphic to $\mathfrak{D}_{s^1-1,s^2} + \{\delta^1_1 \delta^2_2, \delta^2_2 \delta^1_3, \delta^1_3 \delta^2_4, \delta^2_4 \delta^1_1\}$.

	Similarly, suppose to the contrary that for some $s \ge 0$, the graph $\mathfrak{E}_s +\{\varepsilon^1_1 \varepsilon^3_2,\varepsilon^1_3 \varepsilon^3_2,\varepsilon^1_3 \varepsilon^2\}$ contains a $K_{3,4}$ minor. Choose such an $s$ minimum. One can directly verify that $\mathfrak{E}_2 +\{\varepsilon^1_1 \varepsilon^3_2,\varepsilon^1_3 \varepsilon^3_2,\varepsilon^1_3 \varepsilon^2\}$ contains no $K_{3,4}$ minor. Hence $s>2$. Again, we may contract a spine edge to obtain a graph isomorphic to $\mathfrak{E}_{s-1} + \{\varepsilon^1_1 \varepsilon^3_2, \varepsilon^1_3 \varepsilon^3_2, \varepsilon^1_3 \varepsilon^2\}$. By the minimality of $s$, this graph contains no $K_{3,4}$ minor, contradicting Lemma~\ref{lem:to2}.
	
	Finally, one can directly check that the graph $\mathfrak{F}+\{\varphi^1_3 \varphi^2_3\}$ does not contain $K_{3,4}$ as a minor.
\end{proof}

A graph $G$  is \emph{oloidal} if one of the following holds:
\begin{itemize}
	\item $G$ is isomorphic to $\mathfrak{D}_{0,s^2} + A$ with $s^2\in\{0,1\}$, and $A \subseteq \{\delta^1_1 \delta^2_2, \delta^1_1 \delta^2_3, \delta^1_1 \delta^2_4\}$ or $A \subseteq \{\delta^1_1 \delta^2_2, \delta^2_2 \delta^1_3, \delta^1_3 \delta^2_4, \delta^2_4 \delta^1_1\}$.
	\item $G$ is isomorphic to $\mathfrak{D}_{s^1,s^2} + A$ with $s^1,s^2\ge0$, $s^1+s^2\ge2$, and $A \subseteq \{\delta^1_1 \delta^2_2, \delta^2_2 \delta^1_3, \delta^1_3 \delta^2_4, \delta^2_4 \delta^1_1\}$.
	\item $G$ is isomorphic to $\mathfrak{E}_s+A$ with $s\ge0$ and $A\subseteq \{\varepsilon^1_1 \varepsilon^3_2, \varepsilon^1_3 \varepsilon^3_2, \varepsilon^1_3 \varepsilon^2\}$.
	\item $G$ is isomorphic to $\mathfrak{F}+A$ with $A\subseteq \{\varphi^1_3 \varphi^2_3\}$.
\end{itemize}

Note that for $s\ge0$, $\mathfrak{E}_s+\{\varepsilon^1_1 \varepsilon^3_2, \varepsilon^1_3 \varepsilon^3_2, \varepsilon^1_3 \varepsilon^2\}$ is isomorphic to $\mathfrak{D}_{0,s+1} + \{\delta^1_1 \delta^2_2, \delta^1_1 \delta^2_3, \delta^1_1 \delta^2_4\}$.

\begin{proposition}\label{pro:olo}
	Every oloidal graph is $4$-connected and non-projective-planar and contains no $K_{3,4}$ minor.
\end{proposition}
\begin{proof}
	Lemma~\ref{lem:DEF4con} implies that every oloidal graph is $4$-connected and non-projective-planar.
	
	Observe that $\mathfrak{D}_{0,1} + \{\delta^1_1 \delta^2_2, \delta^1_1 \delta^2_3, \delta^1_1 \delta^2_4\}$ is isomorphic to $\mathfrak{E}_0 + \{\varepsilon^1_1 \varepsilon^3_2, \varepsilon^1_3 \varepsilon^3_2, \varepsilon^1_3 \varepsilon^2\}$, where the vertices
	$\delta^1_1, \delta^1_2, \delta^1_3, \delta^1_4, \delta^2_1, \delta^2_2, \delta^2_3, \delta^2_4, \sigma^2_1$
	correspond to
	$\varepsilon^3_2, \varepsilon^3_1, \varepsilon^4, \varepsilon^3_3, \varepsilon^1_2, \varepsilon^1_1, \varepsilon^2, \varepsilon^1_3, \varepsilon^0$,
	respectively. Thus, it follows from Lemma~\ref{lem:DEFK34mf} that every oloidal graph has no $K_{3,4}$ minor.
\end{proof}

We will establish the converse of Proposition~\ref{pro:olo}.

\begin{theorem}\label{thm:4cnppK34mf}
	A graph is $4$-connected, non-projective-planar, and has no $K_{3,4}$ minor if and only if it is oloidal.
\end{theorem}

To prove Theorem~\ref{thm:4cnppK34mf}, we apply Proposition~\ref{pro:DEF}. To this end, in Section~\ref{sec:E20} (respectively, Section~\ref{sec:F4}) we study the structure of $4$-connected graphs with no $K_{3,4}$ minor that contain a spanning JT-subdivision of $E_{20}$ (respectively, $F_4$). Finally, we complete the proof of Theorem~\ref{thm:4cnppK34mf} in Section~\ref{sec:nppolo}.

\subsubsection{Containing a spanning subdivision of $E_{20}$}\label{sec:E20}

In this section we study the structure of $4$-connected graphs that have no $K_{3,4}$ minor, no $F_4$ minor, and no minor of any graph obtained from $E_{20}$ by splitting a vertex, yet contain a spanning JT-subdivision of $E_{20}$. Our purpose is to establish Proposition~\ref{pro:E20}, beginning with a sequence of lemmas, some of which have already been given in~\cite{Lo2025}.

A collection of graphs is presented in Table~\ref{tab:1}. The graph in row $i$ and column $j$ is denoted by $(i,j)$. These graphs are frequently referenced in our arguments. For instance, given a graph $G$ with no $K_{3,4}$ minor that contains a spanning subdivision $\eta(E_{20})$ of $E_{20}$, we can deduce that no vertex in the domain of $e^0$ is adjacent to a vertex in the domain of $e^4$; otherwise, $G \succeq (2,2) \succeq K_{3,4}$, contradicting the assumption. By a slight abuse of notation, we may simply write: ``by~$(2,2)$, no vertex in the domain of $e^0$ is adjacent to a vertex in the domain of $e^4$.''

\begin{table}[ht!]
	\centering
	\caption{A collection of graphs. The graphs in the first row contain $F_4$ as a spanning subgraph, while all others contain $K_{3,4}$ as a minor.}
	\label{tab:1}
	\renewcommand{\arraystretch}{2} 
	\begin{tabular}{ccccccc}
		& {\footnotesize 1} & {\footnotesize 2} & {\footnotesize 3} & {\footnotesize 4} & {\footnotesize 5} & {\footnotesize 6} \\
		\noalign{\bigskip}
		{\footnotesize 1} &
		\includegraphics[scale=1, align=c]{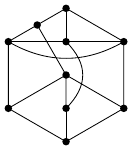} &
		\includegraphics[scale=1, align=c]{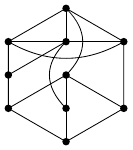} & & & & \\
		\noalign{\bigskip}
		{\footnotesize 2} &
		\includegraphics[scale=1, align=c]{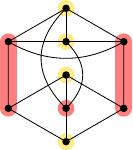} &
		\includegraphics[scale=1, align=c]{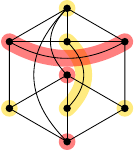} &
		\includegraphics[scale=1, align=c]{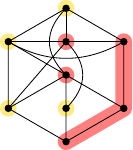} &
		\includegraphics[scale=1, align=c]{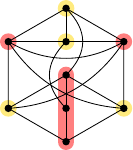} &
		\includegraphics[scale=1, align=c]{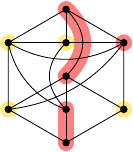} &
		\includegraphics[scale=1, align=c]{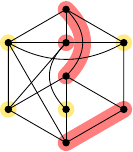} \\
		\noalign{\bigskip}
		{\footnotesize 3} &
		\includegraphics[scale=1, align=c]{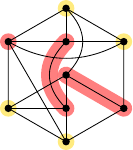} &
		\includegraphics[scale=1, align=c]{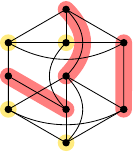} &
		\includegraphics[scale=1, align=c]{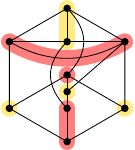} &
		\includegraphics[scale=1, align=c]{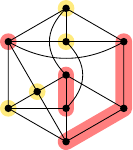} &
		\includegraphics[scale=1, align=c]{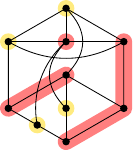} &
		\includegraphics[scale=1, align=c]{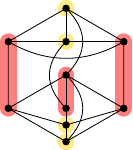} \\
	\end{tabular}
\end{table}

\begin{lemma}[\cite{Lo2025}] \label{lem:E20ni1}
	Let \( G \) be a \( 4 \)-connected graph that contains neither $K_{3,4}$, \( F_4 \), nor any graph obtained from \( E_{20} \) by splitting a vertex as a minor.  
	Suppose \( G \) contains a spanning JT-subdivision \( \eta(E_{20}) \) of \( E_{20} \).  
	Then the segment \( \eta(e) \) has no internal vertex for any  
	\( e \in \{ e^0 e^1_1,\, e^0 e^1_2,\, e^0 e^1_3,\, e^1_1 e^1_2,\, e^1_2 e^1_3,\, e^1_3 e^1_1 \} \).
\end{lemma}

\begin{lemma} \label{lem:E20ni2}
	Let \( G \) be a \( 4 \)-connected graph that contains neither $K_{3,4}$, \( F_4 \), nor any graph obtained from \( E_{20} \) by splitting a vertex as a minor.  
	Suppose \( G \) contains a spanning JT-subdivision \( \eta(E_{20}) \) of \( E_{20} \).  
	Then the segment \( \eta(e) \) has no internal vertex for any  
	\( e \in \{ e^0 e^2,\, e^3_1 e^2,\, e^3_2 e^2,\, e^3_3 e^2 \} \).
\end{lemma}

\begin{proof}
	Suppose, to the contrary, that for some  
	\( e \in \{ e^0 e^2,\, e^3_1 e^2,\, e^3_2 e^2,\, e^3_3 e^2 \} \),  
	the segment \( \eta(e) \) has an internal vertex \( v \).
	
	By Lemma~\ref{lem:JT}, each segment of \( \eta(E_{20}) \) is an induced path in $G$. Since \( G \) is \( 4 \)-connected, the vertex \( v \) has at least two neighbors outside \( \eta(e) \).
	
	\smallskip
	
	\noindent\textbf{Case~1.} \( e = e^0 e^2 \). 
	
	\smallskip
	
	If follows that some neighbor of \( v \) lies in the domain of \( e^1_j \) or \( e^3_j \) for some \( j \in [3] \), or in the domain of \( e^4 \).  
	If a neighbor of \( v \) lies in the domain of \( e^1_j \) with \( j \in [3] \), then \( G \) must contain $(1,1)$ as a minor, since the graph obtained from \( \eta(E_{20}) \) by adding the edge joining \( v \) and its neighbor in the domain of \( e^1_j \), contracting the domain of \( e^1_j \) into a new vertex \( w \), and finally deleting the edge joining \( w \) and \( \eta(e^0) \), is a subdivision of $(1,1)$.  
	(Such details are often omitted if obvious.)  
	Hence \( G \) contains \( F_4 \) as a minor, which is a contradiction.  
	Similarly, if a neighbor of \( v \) lies in the domain of \( e^3_j \) for some \( j \in [3] \), or in the domain of \( e^4 \), then \( G \) must contain a minor of $(2,1)$ or $(2,2)$, and hence a minor of \( K_{3,4} \), again a contradiction.
	
	\smallskip
	
	\smallskip
	
\noindent\textbf{Case~2.} \( e = e^3_i e^2 \) for some \( i \in [3] \).

\smallskip

It follows that any neighbor of \( v \) does not lie in the domain of \( e^0 \) or \( e^1_j \) with \( j \in [3] \setminus \{i\} \); otherwise, \( G \) would contain a minor of $(2,1)$ or $(3,3)$, and hence a \( K_{3,4} \) minor, which is impossible.  
	Similarly, any neighbor of \( v \) does not lie in the domain of \( e^3_j \) with \( j \in [3] \setminus \{i\} \); otherwise, one obtains a subdivision of $(2,1)$ from \( \eta(E_{20}) \) by adding the edge joining \( v \) and the domain of \( e^3_j \), contracting the domain of \( e^3_j \) into a new vertex \( w_1 \), contracting \( \eta(e^2 e^3_{6-i-j}) \) into a new vertex \( w_2 \), and finally deleting the edge joining \( w_1 \) and \( w_2 \).
	
	Thus, any neighbor of \( v \) must lie in  
	\( \eta(e^3_i e^1_i) - \eta(e^3_i) \) or \( \eta(e^3_i e^4) - \eta(e^3_i) \).
	
	In fact, by Lemma~\ref{lem:JT}, the vertex \( v \) has exactly two neighbors outside \( \eta(e) \), namely the neighbors of \( \eta(e^3_i) \) in \( \eta(e^3_i e^1_i) \) and \( \eta(e^3_i e^4) \).  
	Moreover, since \( G \) is \( 4 \)-connected, \( \eta(e^3_i) \) has a neighbor in the domain of \( e^0 \), or in the domain of \( e^1_j \) or \( e^3_j \) for some \( j \in [3] \setminus \{i\} \).  
	This implies that \( G \) contains $(2,1)$, $(2,3)$, or $(3,4)$ as a minor, a contradiction.
	
	\smallskip
	
	This completes the proof.
\end{proof}

\begin{lemma} \label{lem:E20ip}
	Let \( G \) be a \( 4 \)-connected graph that contains neither $K_{3,4}$ nor any graph obtained from \( E_{20} \) by splitting a vertex as a minor.  
	Suppose \( G \) contains a spanning JT-subdivision \( \eta(E_{20}) \) of \( E_{20} \).  
	Then, for any \( i \in [3] \), the union of the segments \( \eta(e^1_i e^3_i) \) and \( \eta(e^3_i e^4) \) is an induced path in \( G \).
\end{lemma}

\begin{proof}
	By Lemma~\ref{lem:JT}, every segment of \( \eta(E_{20}) \) is an induced path in $G$.
	
	Suppose, to the contrary, that for some \( i \in [3] \), the union of the segments  
	\( \eta(e^1_i e^3_i) \) and \( \eta(e^3_i e^4) \) is not an induced path in \( G \).  
	That is, there exists an edge joining a vertex in  
	\( \eta(e^1_i e^3_i) - \eta(e^3_i) \) to a vertex in  
	\( \eta(e^3_i e^4) - \eta(e^3_i) \).
	
	Since \( G \) is \( 4 \)-connected, \( \eta(e^3_i) \) has a neighbor in the domain of \( e^0 \), or in the domain of \( e^1_j \) or \( e^3_j \) for some \( j \in [3] \setminus \{i\} \).  
	This implies that \( G \) contains $(2,1)$, $(2,6)$, or $(3,1)$ as a minor, a contradiction.
\end{proof}

\begin{lemma} \label{lem:E20e4e2}
	Let \( G \) be a \( 4 \)-connected graph that contains neither $K_{3,4}$, \( F_4 \), nor any graph obtained from \( E_{20} \) by splitting a vertex as a minor.
	Suppose \( G \) contains a spanning JT-subdivision \( \eta(E_{20}) \) of \( E_{20} \).  
	Then \( \eta(e^4) \) is adjacent to \( \eta(e^2) \) and has degree four in $G$.
\end{lemma}

\begin{proof}
	By Lemma~\ref{lem:JT}, every segment of \( \eta(E_{20}) \) is an induced path in $G$. 
	 
	Since \( G \) is \( 4 \)-connected, \( \eta(e^4) \) has a neighbor \( u \) outside the closed domain of \( e^4 \). By $(2,2)$, \( u \neq \eta(e^0) \).  
	By Lemmas~\ref{lem:E20ni1}, \ref{lem:E20ni2}, and \ref{lem:E20ip}, we conclude that \( u = \eta(e^2) \).  
	This completes the proof.
\end{proof}

\begin{lemma}[\cite{Lo2025}] \label{lem:E20ind}
	Let \( H \) be a graph obtained from \( E_{20} \) by subdividing two independent edges, each with one new vertex, and then joining these two new vertices.  
	Then \( H \) contains \( K_{3,4} \) or \( F_4 \) as a minor.
\end{lemma}

\begin{lemma} \label{lem:E20int}
	Let \( G \) be a \( 4 \)-connected graph that contains neither $K_{3,4}$, \( F_4 \), nor any graph obtained from \( E_{20} \) by splitting a vertex as a minor.  
	Suppose \( G \) contains a spanning JT-subdivision \( \eta(E_{20}) \) of \( E_{20} \), and \( \eta(e^1_i e^3_i) \) or \( \eta(e^3_i e^4) \), with \( i \in [3] \), contains an internal vertex \( v \). Then \( v \) is adjacent to \( \eta(e^2) \) and \( \eta(e^3_j) \) for some \( j \in [3] \setminus \{i\} \) and has degree four in $G$.
\end{lemma}

\begin{proof}
	By Lemma~\ref{lem:JT}, every segment of \( \eta(E_{20}) \) is an induced path in $G$. 
	
	By Lemma~\ref{lem:E20e4e2}, $\eta(e^2)$ and $\eta(e^4)$ are adjacent.
		
	Without loss of generality, assume \( i = 1 \). So $v$ is an internal vertex of $\eta(e)$, where $e$ is either \( e^1_1 e^3_1 \) or \( e^3_1 e^4 \). Clearly, \( v \) has at least two neighbors outside \( \eta(e) \). Let \( u \) be such a neighbor.

	\smallskip
	
	\noindent\textbf{Case~1.} \( e = e^1_1 e^3_1 \). 
	
	\smallskip
	
	By $(2,1)$ and $(1,2)$, \( u \) does not lie in the domain of \( e^0 \) or \( e^1_j \) with \( j \in \{2,3\} \). Moreover, by Lemmas~\ref{lem:E20ip} and~\ref{lem:E20ind}, \( u \) is not in the domain of \( e^4 \). It follows from Lemma~\ref{lem:E20ni2} that \( u \) must be \( \eta(e^2) \), \( \eta(e^3_2) \), or \( \eta(e^3_3) \).
	
	Suppose \( v \) is adjacent to both \( \eta(e^3_2) \) and \( \eta(e^3_3) \). It is clear that \( \eta(e^3_1) \) has a neighbor \( w \) outside the closed domain of $e^3_1$, and \( w \) is not in the domain of \( e^0 \). Thus \( w \) must lie in the domain of \( e^1_2 \) or \( e^1_3 \), or in the domain of \( e^3_2 \) or \( e^3_3 \), which is impossible by $(2,4)$ and $(3,2)$. Therefore, \( v \) is adjacent to \( \eta(e^2) \) and \( \eta(e^3_j) \) for some \( j \in \{2,3\} \) and has degree four in $G$.

	\smallskip
	
	\smallskip
	
	\noindent\textbf{Case~2.} \( e = e^3_1 e^4 \). 
	
	\smallskip
	
	By $(2,1)$ and $(3,5)$, \( u \) does not lie in the domain of \( e^0 \), \( e^1_2 \), or \( e^1_3 \). Consequently, by Lemma~\ref{lem:E20ip}, \( u \) is not in the domain of \( e^1_1 \). It follows from Lemmas~\ref{lem:JT} and~\ref{lem:E20ni2} that \( u \) must be \( \eta(e^2) \) or the neighbor of \( \eta(e^4) \) in \( \eta(e^3_2 e^4) \) or \( \eta(e^3_3 e^4) \).
	
	In fact, \( v \) is not adjacent to both the neighbors of \( \eta(e^4) \) in \( \eta(e^3_2 e^4) \) and in \( \eta(e^3_3 e^4) \), by $(3,6)$. This implies that \( v \) has degree four and is adjacent to \( \eta(e^2) \) and the neighbor of \( \eta(e^4) \) in \( \eta(e^4 e^3_j) \) for some \( j \in \{2,3\} \). Without loss of generality, assume \( j = 3 \).  Denote by $w$ the neighbor of \( \eta(e^4) \) in \( \eta(e^4 e^3_3) \).
	
	We claim that \( w = \eta(e^3_3) \); in other words, the segment \( \eta(e^4 e^3_3) \) has no internal vertex. Suppose otherwise. By the above, \( w \) is adjacent to \( \eta(e^2) \). Note that, by Lemma~\ref{lem:JT}, $v$ and $w$ are neighbors of $\eta(e^4)$ in $\eta(e^3_1 e^4)$ and $\eta(e^3_3 e^4)$, respectively. Moreover, \( \eta(e^3_2) \) is adjacent to neither \( v \) nor \( w \) (as the neighbors of \( v \) and \( w \) have already been determined).  
	
	It is clear that \( \eta(e^3_2) \) has a neighbor outside the closed domain of \( e^3_2 \), and, by $(2,1)$, this neighbor does not lie in the domain of \( e^0 \). By symmetry, it suffices to consider the cases where \( \eta(e^3_2) \) joins the domain of \( e^1_1 \) or the domain of \( e^3_1 \).
	
If \( \eta(e^3_2) \) joins the domain of \( e^1_1 \), then by Lemma~\ref{lem:E20ni1},  
\( \eta(e^3_2) \) is adjacent to a vertex in \( \eta(e^1_1 e^3_1) - \eta(e^3_1) \).  
It is clear that \( \eta(e^3_1) \) has a neighbor outside the closed domain of \( e^3_1 \).  
By \((2,1)\), \((2,4)\), \((2,5)\), and \((3,2)\), this neighbor of \( \eta(e^3_1) \) does not lie in the domain of  
\( e^0 \), \( e^1_2 \), \( e^1_3 \), or \( e^3_2 \).  
If that neighbor lies in \( \eta(e^3_3 e^4) - \{w, \eta(e^4)\} \), then \( G \) contains a minor of \((2,5)\) (this can be seen by contracting the union of \( \eta(e^1_3 e^3_3) \) and \( \eta(e^3_3 e^4) - \{w, \eta(e^4)\} \)).  
As \( \eta(e^3_1) \) and \( w \) are not adjacent, we have considered all cases.  
In each case, we obtain a contradiction to the assumption that \( G \) has no \( K_{3,4} \) minor.

	If \( \eta(e^3_2) \) joins the domain of \( e^3_1 \), then \( G \) would contain a minor of $(3,2)$, a contradiction.
	
	\smallskip
	
	This completes the proof.
\end{proof}

\begin{proposition} \label{pro:E20}
	Let \( G \) be a \( 4 \)-connected graph that contains neither $K_{3,4}$, \( F_4 \), nor any graph obtained from \( E_{20} \) by splitting a vertex as a minor.  
	Suppose \( G \) contains a spanning JT-subdivision \( \eta(E_{20}) \) of \( E_{20} \).
	Then one of the following holds:
	\begin{itemize}
		\item \( G \) contains \( \mathfrak{D}_{0,1} \) as a spanning subgraph such that each of $\delta^1_2$, $\delta^1_4$, and \( \sigma^2_1 \) has degree four in $G$.
		\item \( G \) contains \( \mathfrak{D}_{0,s^2} \), with $s^2=|V(G)|-8\ge2$, as a spanning subgraph such that every vertex on either spine has degree four in $G$.
		\item \( G \) contains \( \mathfrak{E}_0 \) as a spanning subgraph such that each of \( \varepsilon^0 \) and \( \varepsilon^4 \) has degree four in $G$.
	\end{itemize}
\end{proposition}

\begin{proof}
	Recall that, by Lemma~\ref{lem:JT}, every segment of $\eta(E_{20})$ is an induced path and, by Lemma~\ref{lem:E20e4e2}, $\eta(e^4)$ has degree four and is adjacent to $\eta(e^2)$.
	
We consider two cases, depending on whether \( \eta(E_{20}) \) has any internal vertices.

	\smallskip
	
	\noindent\textbf{Case~1.} \( |V(G)| = |V(E_{20})| \). 
	
	\smallskip

	In this case, there is no internal vertex.
	
	By $(2,1)$, each \( \eta(e^3_i) \) with \( i \in [3] \) has a neighbor that is \( \eta(e^1_j) \) or \( \eta(e^3_j) \) for some \( j \in [3] \setminus \{i\} \).

\smallskip

\noindent\textbf{Case~1.1.} There exist distinct \( i, j \in [3] \) such that \( \eta(e^3_i) \) is adjacent to \( \eta(e^1_j) \).

\smallskip

Without loss of generality, assume \( i = 1 \) and \( j = 2 \). By \((2,4)\) and \((2,5)\), \( \eta(e^3_2) \) is adjacent to \( \eta(e^3_3) \). Hence \( G \) contains \( \mathfrak{E}_0 \) as a spanning subgraph. More precisely, \( \mathfrak{E}_0 \) is isomorphic to the graph obtained from \( \eta(E_{20}) \) by adding the edges \( \eta(e^3_1)\eta(e^1_2) \), \( \eta(e^3_2)\eta(e^3_3) \), and \( \eta(e^2)\eta(e^4) \), where the vertices $\varepsilon^0$, $\varepsilon^1_1$, $\varepsilon^1_2$, $\varepsilon^1_3$, $\varepsilon^2$, $\varepsilon^3_1$, $\varepsilon^3_2$, $\varepsilon^3_3$, $\varepsilon^4$ correspond to $\eta(e^1_1)$, $\eta(e^1_2)$, $\eta(e^0)$, $\eta(e^1_3)$, $\eta(e^3_1)$, $\eta(e^3_2)$, $\eta(e^2)$, $\eta(e^3_3)$, $\eta(e^4)$, respectively. 

The vertex \( \varepsilon^4 \), corresponding to \( \eta(e^4) \), has degree four in $G$ by Lemma~\ref{lem:E20e4e2}. 

Observe that \( \mathfrak{E}_0 \) has an automorphism mapping $\varepsilon^0$, $\varepsilon^1_1$, $\varepsilon^1_2$, $\varepsilon^1_3$, $\varepsilon^2$, $\varepsilon^3_1$, $\varepsilon^3_2$, $\varepsilon^3_3$, $\varepsilon^4$ to $\varepsilon^4$, $\varepsilon^3_2$, $\varepsilon^3_1$, $\varepsilon^3_3$, $\varepsilon^2$, $\varepsilon^1_2$, $\varepsilon^1_1$, $\varepsilon^1_3$, $\varepsilon^0$, respectively. Consequently, by Proposition~\ref{pro:JT}, we may consider another spanning JT-subdivision \( \eta'(E_{20}) \) of \( E_{20} \) with \( \eta'(e^4)=\eta(e^1_1) \), and conclude that \( \varepsilon^0 \), corresponding to \( \eta(e^1_1) \), has degree four in $G$.

\smallskip

\noindent\textbf{Case~1.2.} There do not exist distinct \( i, j \in [3] \) such that \( \eta(e^3_i) \) is adjacent to \( \eta(e^1_j) \).

\smallskip

	Without loss of generality, we assume $\eta(e^3_1)\eta(e^3_2), \eta(e^3_2)\eta(e^3_3) \in E(G)$. In fact, \( \mathfrak{D}_{0,1} \) is isomorphic to the graph obtained from \( \eta(E_{20}) \) by adding the edges $\eta(e^3_1)\eta(e^3_2)$, $\eta(e^3_2)\eta(e^3_3)$, and \( \eta(e^2)\eta(e^4) \), where the vertices $\delta^1_1$, $\delta^1_2$, $\delta^1_3$, $\delta^1_4$, $\delta^2_1$, $\delta^2_2$, $\delta^2_3$, $\delta^2_4$, $\sigma^2_1$ correspond to $\eta(e^1_1)$, $\eta(e^1_2)$, $\eta(e^1_3)$, $\eta(e^0)$, $\eta(e^3_1)$, $\eta(e^3_2)$, $\eta(e^3_3)$, $\eta(e^2)$, $\eta(e^4)$, respectively. So \( G \) contains \( \mathfrak{D}_{0,1} \) as a spanning subgraph. 
	
	By Lemma~\ref{lem:E20e4e2}, the vertex $\sigma^2_1$, corresponding to $\eta(e^4)$, has degree four in $G$. 
	
By $(2,1)$ and $(2,2)$, the vertex $\delta^1_4$, corresponding to $\eta(e^0)$, has degree four in $G$. Note that the subgraph of $G$ isomorphic to $\mathfrak{D}_{0,1}$ admits an automorphism that swaps $\eta(e^0)$ and $\eta(e^1_2)$, swaps $\eta(e^2)$ and $\eta(e^3_2)$, and fixes all other vertices. Therefore, again by $(2,1)$ and $(2,2)$, we deduce that the vertex $\delta^1_2$, corresponding to $\eta(e^1_2)$, also has degree four in $G$.

	\smallskip

	\smallskip
	
	\noindent\textbf{Case~2.} \( |V(G)| > |V(E_{20})| \). 
	
	\smallskip
	
By Lemmas~\ref{lem:E20ni1},~\ref{lem:E20ni2}, and~\ref{lem:E20int}, there exist distinct \( i, j \in [3] \) such that \( \eta(e) \), where \( e = e^1_i e^3_i \) or \( e^3_i e^4 \), has an internal vertex \( v \) adjacent to $\eta(e^2)$ and \( \eta(e^3_j) \). Without loss of generality, assume \( i = 1 \) and \( j = 2 \).

	\smallskip
	
	\noindent\textbf{Case~2.1.} $e = e^1_1 e^3_1$.
	
	\smallskip

	By $(2,4)$, $(2,5)$, and Lemma~\ref{lem:E20int}, \( \eta(e^1_2 e^3_2) \) has no internal vertex.
	
	If \( \eta(e^3_2 e^4) \) had an internal vertex, then by Lemma~\ref{lem:E20int}, that internal vertex would be adjacent to $\eta(e^2)$ and \( G \) would contain as a minor the graph obtained from \( E_{20} \) by subdividing \( e^1_1 e^3_1 \) and \( e^1_2 e^3_2 \), each with one new vertex, and joining these new vertices with an edge. This yields a contradiction by Lemma~\ref{lem:E20ind}. Hence \( \eta(e^3_2 e^4) \) has no internal vertex.
	
	By $(3,2)$ and Lemma~\ref{lem:E20int}, it is straightforward to show that any internal vertex of \( \eta(e^1_1 e^3_1) \), \( \eta(e^3_1 e^4) \), \( \eta(e^1_3 e^3_3) \), or \( \eta(e^3_3 e^4) \) is adjacent to \( \eta(e^3_2) \). 
	
By $(2,1)$, $(2,4)$, $(2,5)$, $(3,2)$, Lemma~\ref{lem:E20ni2}, and the fact that \( \eta(e^1_2 e^3_2) \) and \( \eta(e^3_2 e^4) \) have no internal vertex, we have that \( \eta(e^3_1) \) has degree four and is adjacent to \( \eta(e^3_2) \).

By $(3,2)$ and by the fact that no internal vertex of \( \eta(e^1_1 e^3_1) \) is adjacent to \( \eta(e^3_3) \), \( \eta(e^3_3) \) does not join the domain of \( e^3_1 \). By $(2,1)$, \( \eta(e^3_3) \) does not join the domain of \( e^0 \). By $(1,2)$ and by the fact that no internal vertex of \( \eta(e^1_1 e^3_1) \) is adjacent to \( \eta(e^3_3) \), we have that \( \eta(e^3_3) \) joins neither the domain of \( e^1_1 \) nor that of \( e^1_2 \). (To see this, consider another subdivision \( \eta'(E_{20}) \) of \( E_{20} \) such that \( \eta'(e^3_1) = v \), \( \eta'(e^4) = \eta(e^3_1) \), and \( \eta'(e^3_3) = \eta(e^4) \); then \( \eta(e^3_3) \) becomes an internal vertex of \( \eta'(e^1_3 e^3_3) \).) Therefore, by Lemma~\ref{lem:E20ni2} and the fact that \( \eta(e^1_2 e^3_2) \) and \( \eta(e^3_2 e^4) \) have no internal vertex, \( \eta(e^3_3) \) has degree four and is adjacent to \( \eta(e^3_2) \).

	Thus, $G$ contains $\mathfrak{D}_{0,s^2}$, where $s^2 = |V(G)| - 8 \ge 2$, as a spanning subgraph. This subgraph is obtained from $\eta(E_{20})$ by adding the edges $\eta(e^2)\eta(e^4)$, $\eta(e^3_1)\eta(e^3_2)$, $\eta(e^3_2)\eta(e^3_3)$, and the edges incident to internal vertices. The vertices $\delta^1_1$, $\delta^1_2$, $\delta^1_3$, $\delta^1_4$, $\delta^2_1$, $\delta^2_2$, $\delta^2_3$, $\delta^2_4$ correspond to \( \eta(e^1_1) \), \( \eta(e^1_2) \), \( \eta(e^1_3) \), \( \eta(e^0) \), the neighbor of \( \eta(e^1_1) \) in \( \eta(e^1_1 e^3_1) \), \( \eta(e^3_2) \), the neighbor of \( \eta(e^1_3) \) in \( \eta(e^1_3 e^3_3) \), \( \eta(e^2) \), respectively, and the non-trivial spine, with end-vertices \( \delta^2_1 \) and \( \delta^2_3 \), corresponds to the union of \( \eta(e^1_1 e^3_1) - \eta(e^1_1) \), \( \eta(e^3_1 e^4) \), \( \eta(e^4 e^3_3) \), and \( \eta(e^3_3 e^1_3) - \eta(e^1_3) \). 
	
	As shown above, every vertex of the non-trivial spine has degree four.

Recall that no internal vertex is adjacent to $\eta(e^0)$ or $\eta(e^1_2)$. Note that the subgraph isomorphic to $\mathfrak{D}_{0,s^2}$ admits an automorphism that swaps $\eta(e^0)$ and $\eta(e^1_2)$, swaps $\eta(e^2)$ and $\eta(e^3_2)$, and fixes all other vertices. Hence, by $(2,1)$ and $(2,2)$, we conclude that the vertices $\delta^1_4$ and $\delta^1_2$ of the trivial spine, corresponding to $\eta(e^0)$ and $\eta(e^1_2)$, respectively, both have degree four.

\smallskip

\noindent\textbf{Case~2.2.} None of \( \eta(e^1_1 e^3_1) \), \( \eta(e^1_2 e^3_2) \), or \( \eta(e^1_3 e^3_3) \) has an internal vertex, and $e = e^1_1 e^3_1$.

\smallskip

Most arguments are similar to those in Case~2.1. 

Recall that \( v \) is an internal vertex of \( \eta(e^3_1 e^4) \) and is adjacent to \( \eta(e^3_2) \).

	It follows from Lemmas~\ref{lem:E20int} and~\ref{lem:E20ind} that \( \eta(e^3_2 e^4) \) has no internal vertex. Moreover, by Lemma~\ref{lem:E20int} and $(3,2)$, any internal vertex of \( \eta(e^3_1 e^4) \) or \( \eta(e^3_3 e^4) \) is adjacent to \( \eta(e^3_2) \).
	
	Using $(2,1)$, $(1,2)$, and $(3,2)$, one can deduce that both \( \eta(e^3_1) \) and \( \eta(e^3_3) \) have degree four and are adjacent to \( \eta(e^3_2) \).
	
Therefore, we conclude that \( G \) contains \( \mathfrak{D}_{0,s^2} \) with \( s^2 = |V(G)| - 8 \ge 2 \) as a spanning subgraph, such that $\delta^1_1$, $\delta^1_2$, $\delta^1_3$, $\delta^1_4$, $\delta^2_1$, $\delta^2_2$, $\delta^2_3$, $\delta^2_4$ correspond to \( \eta(e^1_1) \), \( \eta(e^1_2) \), \( \eta(e^1_3) \), \( \eta(e^0) \), \( \eta(e^3_1) \), \( \eta(e^3_2) \), \( \eta(e^3_3) \), \( \eta(e^2) \), respectively, and the non-trivial spine, with end-vertices \( \delta^2_1 \) and \( \delta^2_3 \), corresponds to the union of \( \eta(e^3_1 e^4) \) and \( \eta(e^4 e^3_3) \). 

Every vertex on a non-trivial spine has degree four. Moreover, by the same argument as in Case~2.1, the vertices on the trivial spine also have degree four.
\end{proof}

\subsubsection{Containing a spanning subdivision of $F_4$}\label{sec:F4}

In this section we analyze $4$-connected graphs with no $K_{3,4}$ minor, with no minor of any graph obtained from $F_4$ by splitting a vertex, and containing a spanning JT-subdivision of $F_4$. Our goal is to prove Proposition~\ref{pro:F4}; to this end, we establish a sequence of lemmas.

Analogously to the previous section, we prepare a collection of graphs in Table~\ref{tab:2}, using the same conventions as before.

\begin{table}[ht!]
	\centering
	\caption{A collection of graphs, each containing $K_{3,4}$ as a minor.}
	\label{tab:2}
	\renewcommand{\arraystretch}{2} 
	\begin{tabular}{ccccccc}
		& {\footnotesize 1} & {\footnotesize 2} & {\footnotesize 3} & {\footnotesize 4} & {\footnotesize 5} & {\footnotesize 6} \\
		\noalign{\bigskip}
		{\footnotesize 4} &
		\includegraphics[scale=1, align=c]{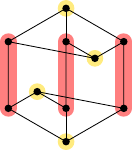} &
		\includegraphics[scale=1, align=c]{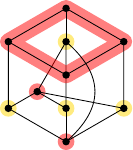} &
		\includegraphics[scale=1, align=c]{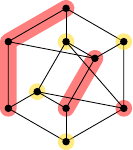} &
		\includegraphics[scale=1, align=c]{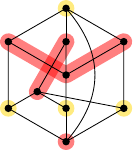} &
		\includegraphics[scale=1, align=c]{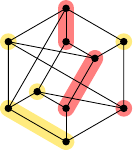} &
		\includegraphics[scale=1, align=c]{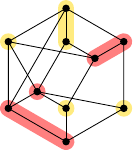} \\
		\noalign{\bigskip}
		{\footnotesize 5} &
		\includegraphics[scale=1, align=c]{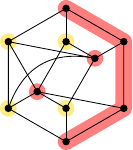} &
		\includegraphics[scale=1, align=c]{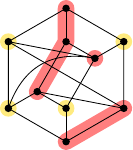} &
		\includegraphics[scale=1, align=c]{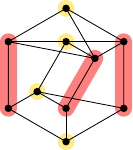} &
		\includegraphics[scale=1, align=c]{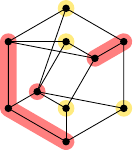} &
		\includegraphics[scale=1, align=c]{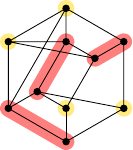} &
		\includegraphics[scale=1, align=c]{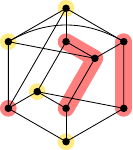} \\
		\noalign{\bigskip}
		{\footnotesize 6} &
		\includegraphics[scale=1, align=c]{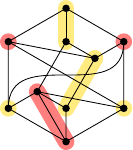} &
		\includegraphics[scale=1, align=c]{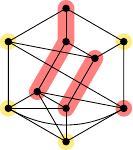} &
		\includegraphics[scale=1, align=c]{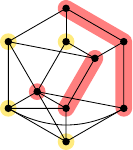} &
		\includegraphics[scale=1, align=c]{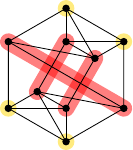} &
		\includegraphics[scale=1, align=c]{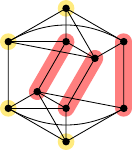} &
		\includegraphics[scale=1, align=c]{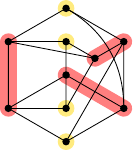} \\
		\noalign{\bigskip}
		{\footnotesize 7} &
		\includegraphics[scale=1, align=c]{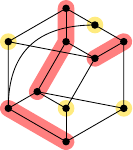} &
		\includegraphics[scale=1, align=c]{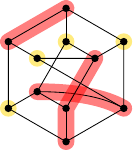} &
		\includegraphics[scale=1, align=c]{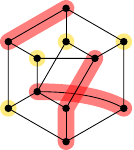} &
		\includegraphics[scale=1, align=c]{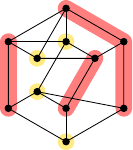} &
		\includegraphics[scale=1, align=c]{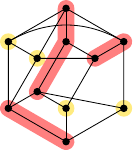} &
		\includegraphics[scale=1, align=c]{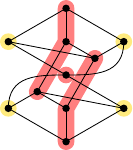} \\
		\noalign{\bigskip}
		{\footnotesize 8} &
		\includegraphics[scale=1, align=c]{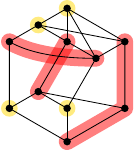} &
		\includegraphics[scale=1, align=c]{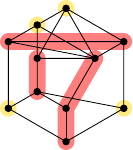} & & & & \\
	\end{tabular}
\end{table}

\begin{lemma} \label{lem:F4LR}
	Let \( G \) be a \( 4 \)-connected graph that contains neither \( K_{3,4} \) nor any graph obtained from \( F_4 \) by splitting a vertex as a minor.  
	Suppose \( G \) contains a spanning JT-subdivision \( \eta(F_4) \) of \( F_4 \).  
	For any \( i \in [2] \), \( \eta(f^i_2) \) joins to at least one of \( \eta(f^i_1 f^i) - \eta(f^i) \), \( \eta(f^i_1 f^i_3) - \eta(f^i_3) \), \( \eta(f^i_4 f^i) - \eta(f^i) \), or \( \eta(f^i_4 f^i_3) - \eta(f^i_3) \).
\end{lemma}

\begin{proof}
	By Lemma~\ref{lem:JT}, every segment of \( \eta(F_4) \) is an induced path in $G$. As $G$ is $4$-connected, \( \eta(f^i_2) \) has a neighbor outside the closed domain of \( f^i_2 \). 
	
By \( (4,1) \), \( (4,2) \), and \( (4,3) \), 
\( \eta(f^i_2) \) must join to one of 
\( \eta(f^i_1 f^i) - \eta(f^i) \), 
\( \eta(f^i_1 f^i_3) - \eta(f^i_3) \), 
\( \eta(f^i_4 f^i) - \eta(f^i) \), or 
\( \eta(f^i_4 f^i_3) - \eta(f^i_3) \).
\end{proof}

\begin{lemma} \label{lem:F4nooppo}
	Let \( G \) be a \( 4 \)-connected graph that contains neither \( K_{3,4} \) nor any graph obtained from \( F_4 \) by splitting a vertex as a minor.  
	Suppose \( G \) contains a spanning JT-subdivision \( \eta(F_4) \) of \( F_4 \).  
	For any \( i \in [2] \), no vertex in the domain of \( f^i \) joins to any vertex in the domain of \( f^{3-i} \) or in the domain of \( f^{3-i}_j \) with \( j \in [4] \).
\end{lemma}

\begin{proof}
	Let \( v \) be a vertex in the domain of \( f^i \). 
	
	By $(4,2)$ and $(4,4)$, \( v \) does not join to the domain of \( f^{3-i} \) or \( f^{3-i}_2 \).

	We claim that \( v \) does not join the domain of \( f^{3-i}_3 \). 
	Suppose otherwise. If \( v \) were an internal vertex of \( \eta(f^1 f^1_1) \), 
	then \( G \) would contain a minor of \( (5,4) \) (to see this, contract 
	\( \eta(f^1 f^1_4) \)), a contradiction. Thus \( v \) is not an internal vertex 
	of \( \eta(f^1 f^1_1) \), and, by symmetry, not of \( \eta(f^1 f^1_4) \). 
	Hence \( v \) lies in \( \eta(f^1 f^1_2) - \eta(f^1_2) \). 
	By Lemma~\ref{lem:F4LR}, \( G \) would then again contain a minor of \( (5,4) \), 
	a contradiction.

By symmetry, it remains to show that \( v \) does not join any vertex in \( \eta(f^i_1 f^{3-i}_4) - \eta(f^i_1) \). Suppose otherwise. As \( G \) contains neither \( (4,3) \) nor \( (7,1) \) as a minor, \( v \) lies in \( \eta(f^i f^i_1) - \eta(f^i_1) \).

	Clearly, \( \eta(f^i_1) \) has a neighbor \( u \) outside the closed domain of \( f^i_1 \). By $(5,5)$, $(5,6)$, $(4,5)$, $(4,3)$, and $(4,6)$, \( u \) is not in the domain of \( f^i_2 \), \( f^i_4 \), \( f^{3-i}_1 \), \( f^{3-i}_2 \), or \( f^{3-i}_3 \). Thus \( u \) lies in \( \eta(f^{3-i} f^{3-i}_4) - \eta(f^{3-i}_4) \).

	By Lemma~\ref{lem:F4LR}, \( \eta(f^i_2) \) (respectively, \( \eta(f^{3-i}_2) \)) has a neighbor either in  
	\( \eta(f^i_1 f^i) - \eta(f^i) \) or \( \eta(f^i_1 f^i_3) - \eta(f^i_3) \) (respectively, in \( \eta(f^{3-i}_1 f^{3-i}) - \eta(f^{3-i}) \) or \( \eta(f^{3-i}_1 f^{3-i}_3) - \eta(f^{3-i}_3) \)),  
	or in  
	\( \eta(f^i_4 f^i) - \eta(f^i) \) or \( \eta(f^i_4 f^i_3) - \eta(f^i_3) \) (respectively, in \( \eta(f^{3-i}_4 f^{3-i}) - \eta(f^{3-i}) \) or \( \eta(f^{3-i}_4 f^{3-i}_3) - \eta(f^{3-i}_3) \)).
	
	By $(5,5)$, if \( \eta(f^i_2) \) (respectively, \( \eta(f^{3-i}_2) \)) has a neighbor in  
	\( \eta(f^i_1 f^i) - \eta(f^i) \) or \( \eta(f^i_1 f^i_3) - \eta(f^i_3) \) (respectively, in \( \eta(f^{3-i}_4 f^{3-i}) - \eta(f^{3-i}) \) or \( \eta(f^{3-i}_4 f^{3-i}_3) - \eta(f^{3-i}_3) \)),  
	then \( v \) (repsectively, $u$) must be an internal vertex of \( \eta(f^i f^i_1) \) (respectively, of \( \eta(f^{3-i} f^{3-i}_4) \)). 
	
We use the above observation to analyze the different cases according to the neighbors of \( \eta(f^i_2) \) and \( \eta(f^{3-i}_2) \).

If \( \eta(f^i_2) \) has a neighbor in \( \eta(f^i_1 f^i) - \eta(f^i) \) or \( \eta(f^i_1 f^i_3) - \eta(f^i_3) \), and \( \eta(f^{3-i}_2) \) has a neighbor in \( \eta(f^{3-i}_4 f^{3-i}) - \eta(f^{3-i}) \) or \( \eta(f^{3-i}_4 f^{3-i}_3) - \eta(f^{3-i}_3) \), then \( v \) is an internal vertex of \( \eta(f^i f^i_1) \) and \( u \) is an internal vertex of \( \eta(f^{3-i} f^{3-i}_4) \); hence \( G \) contains a minor of \( (6,6) \) (to see this, contract \( \eta(f^i f^i_4) \) and \( \eta(f^{3-i} f^{3-i}_1) \)), a contradiction.

If \( \eta(f^i_2) \) has a neighbor in \( \eta(f^i_1 f^i) - \eta(f^i) \) or \( \eta(f^i_1 f^i_3) - \eta(f^i_3) \), and \( \eta(f^{3-i}_2) \) has a neighbor in \( \eta(f^{3-i}_1 f^{3-i}) - \eta(f^{3-i}) \) or \( \eta(f^{3-i}_1 f^{3-i}_3) - \eta(f^{3-i}_3) \), then \( G \) also contains a minor of \( (6,6) \) (to see this, contract \( \eta(f^i f^i_4) \)), a contradiction. The case where \( \eta(f^i_2) \) has a neighbor in \( \eta(f^i_4 f^i) - \eta(f^i) \) or \( \eta(f^i_4 f^i_3) - \eta(f^i_3) \), and \( \eta(f^{3-i}_2) \) has a neighbor in \( \eta(f^{3-i}_4 f^{3-i}) - \eta(f^{3-i}) \) or \( \eta(f^{3-i}_4 f^{3-i}_3) - \eta(f^{3-i}_3) \), can be handled similarly. 

Finally, if \( \eta(f^i_2) \) has a neighbor in \( \eta(f^i_4 f^i) - \eta(f^i) \) or \( \eta(f^i_4 f^i_3) - \eta(f^i_3) \), and \( \eta(f^{3-i}_2) \) has a neighbor in \( \eta(f^{3-i}_1 f^{3-i}) - \eta(f^{3-i}) \) or \( \eta(f^{3-i}_1 f^{3-i}_3) - \eta(f^{3-i}_3) \), then one directly sees that \( G \) contains a minor of \( (6,6) \), a contradiction.

This completes the proof.
\end{proof}

\begin{lemma} \label{lem:F4noint}
	Let \( G \) be a \( 4 \)-connected graph that contains neither \( K_{3,4} \) nor any graph obtained from \( F_4 \) by splitting a vertex as a minor.  
	Suppose \( G \) contains a spanning JT-subdivision \( \eta(F_4) \) of \( F_4 \).  
	Then the segment \( \eta(e) \) has no internal vertex for any  
	\( e \in \{ f^1_1 f^1_3, f^1_2 f^1_3, f^1_4 f^1_3, f^2_1 f^2_3, f^2_2 f^2_3, f^2_4 f^2_3 \} \).
\end{lemma}

\begin{proof}
Suppose to the contrary that \( \eta(e) \) contains an internal vertex \( v \). Clearly, \( v \) has at least two neighbors outside \( \eta(e) \).

	By symmetry, it suffices to treat the cases \( e = f^1_1 f^1_3 \) and \( e = f^1_2 f^1_3 \).

	\smallskip
	
	\noindent\textbf{Case~1.} \( e = f^1_1 f^1_3 \).
	
	\smallskip
	
The vertex \( v \) joins neither the domain of \( f^1_2 \) nor \( f^1_4 \) since \( G \) does not contain, as a minor, any graph obtained from \( F_4 \) by splitting the vertex $f^1_3$. Moreover, \( v \) does not join the domain of any of \( f^2_1, f^2_2, f^2_3 \) by $(7,2)$, $(4,3)$, and $(7,3)$. By Lemma~\ref{lem:F4nooppo}, \( v \) also does not join the domain of \( f^2 \).

	Thus, by Lemma~\ref{lem:JT} and the fact that $G$ is $4$-connected, \( v \) has exactly two neighbors outside \( \eta(e) \), namely the neighbor of \( \eta(f^1_1) \) in \( \eta(f^1_1 f^1) \) and  
	the neighbor of \( \eta(f^1_1) \) in \( \eta(f^1_1 f^2_4) \).
	
It is clear that \( \eta(f^1_1) \) has a neighbor \( u \) outside the closed domain of \( f^1_1 \). By $(7,4)$, $(7,5)$, $(5,2)$, $(4,3)$, and $(5,1)$, \( u \) does not lie in the domain of \( f^1_2 \), \( f^1_4 \), \( f^2_1 \), \( f^2_2 \), or \( f^2_3 \). Hence \( u \) must lie in the domain of \( f^2 \), which contradicts Lemma~\ref{lem:F4nooppo}.

	\smallskip
	
	\smallskip
	
	\noindent\textbf{Case~2.} \( e = f^1_2 f^1_3 \).

	\smallskip
	
	As \( G \) does not contain as a minor a graph obtained from \( F_4 \) by splitting  $f^1_3$, $v$  does not join the domain of \( f^1_1 \) or \( f^1_4 \). By $(4,2)$, $(4,3)$, and $(4,1)$, \( v \) does not join the domain of \( f^2 \), \( f^2_1 \), \( f^2_2 \), or \( f^2_4 \). 
	
	By Lemma~\ref{lem:JT} and the fact that $G$ is $4$-connected,  \( v \) has exactly two neighbors outside \( \eta(e) \), namely the neighbor of \( \eta(f^1_2) \) in \( \eta(f^1_2 f^1) \) and the neighbor of \( \eta(f^1_2) \) in \( \eta(f^1_2 f^2_3) \).
	
By Lemma~\ref{lem:F4LR}, \( \eta(f^1_2) \) joins the union of \( \eta(f^1_j f^1) - \eta(f^1) \) and \( \eta(f^1_j f^1_3) - \eta(f^1_3) \) for some \( j \in \{1,4\} \). This implies that \( G \) contains a minor of $(5,4)$ (to see this, contract \( \eta(f^1 f^1_{5-j}) \)), a contradiction.
\end{proof}

\begin{lemma} \label{lem:F403}
	Let \( G \) be a \( 4 \)-connected graph that contains neither \( K_{3,4} \) nor any graph obtained from \( F_4 \) by splitting a vertex as a minor.  
	Suppose \( G \) contains a spanning JT-subdivision \( \eta(F_4) \) of \( F_4 \).  
	Then for any \( i \in [2] \), \( \eta(f^i) \) is adjacent to \( \eta(f^i_3) \) and has degree four in $G$.
\end{lemma}

\begin{proof}
	It is clear that \( \eta(f^i) \) has a neighbor outside the closed domain of \( f^i \).  
	By Lemma~\ref{lem:F4nooppo} and Lemma~\ref{lem:F4noint}, the only possible such neighbor is \( \eta(f^i_3) \). Thus \( \eta(f^i) \) has degree four and is adjacent to \( \eta(f^i_3) \).
\end{proof}

\begin{lemma} \label{lem:F4int}
	Let \( G \) be a \(4\)-connected graph that contains neither \( K_{3,4} \) nor any graph obtained from \( F_4 \) by splitting a vertex as a minor. Suppose \( G \) contains a spanning JT-subdivision \( \eta(F_4) \) of \( F_4 \), and that \( \eta(f^i f^i_j) \), with \( i \in [2] \) and \( j \in \{1,4\} \), contains an internal vertex \( v \). Then \( \eta(f^i f^i_{5-j}) \) has no internal vertex, and \( v \) is adjacent to \( \eta(f^i_{5-j}) \) and \( \eta(f^i_3) \) and has degree four in $G$.
\end{lemma}

\begin{proof}
It suffices to prove the case where \( v \) lies in \( \eta(e) \) with \( e = f^1 f^1_1 \).

Recall that $\eta(f^1) \eta(f^1_3), \eta(f^2) \eta(f^2_3) \in E(G)$ by  Lemma~\ref{lem:F403}.
	
	By Lemmas~\ref{lem:JT}, \ref{lem:F4nooppo}, and \ref{lem:F4noint}, \( v \) has at least two neighbors outside \( \eta(e) \), each of which is either the neighbor of \( \eta(f^1) \) in \( \eta(f^1 f^1_2) \), the neighbor of \( \eta(f^1) \) in \( \eta(f^1 f^1_4) \), or \( \eta(f^1_3) \).
	
By $(8,1)$, \( v \) is not adjacent to both neighbors of \( \eta(f^1) \) in \( \eta(f^1 f^1_2) \) and in \( \eta(f^1 f^1_4) \). Hence \( v \) has exactly two neighbors outside \( \eta(e) \), and one of them is \( \eta(f^1_3) \).

\smallskip

\noindent\textbf{Case~1.} \( v \) is adjacent to the neighbor of \( \eta(f^1) \) in \( \eta(f^1 f^1_2) \).

\smallskip

We show that this case cannot occur.

It is clear that \( \eta(f^1_1) \) (respectively, \( \eta(f^1_4) \)) has a neighbor outside the closed domain of \( f^1_1 \) (respectively, \( f^1_4 \)). By $(5,3)$, $(8,2)$, $(4,3)$, and Lemma~\ref{lem:F4nooppo}, \( \eta(f^1_1) \) does not join to any domain of \( f^1_2 \), \( f^1_4 \), \( f^2_2 \), or \( f^2 \). Thus, using Lemma~\ref{lem:F4noint}, we conclude that \( \eta(f^1_1) \) joins to \( \eta(f^2_1) \) or \( \eta(f^2_3) \). By $(8,1)$ and $(8,2)$, \( \eta(f^1_4) \) does not join to the domain of \( f^1_1 \). By $(5,3)$, $(4,3)$, and Lemma~\ref{lem:F4nooppo}, \( \eta(f^1_4) \) does not join to any domain of \( f^1_2 \), \( f^2_2 \), or \( f^2 \). It then follows from Lemma~\ref{lem:F4noint} that \( \eta(f^1_4) \) joins to \( \eta(f^2_3) \) or \( \eta(f^2_4) \). Hence we may consider the cases depending on where \( \eta(f^1_1) \) and \( \eta(f^1_4) \) join.

\smallskip

\noindent\textbf{Case~1.1.} Both \( \eta(f^1_1) \) and \( \eta(f^1_4) \) join to \( \eta(f^2_3) \).

\smallskip

By Lemmas~\ref{lem:F4LR} and~\ref{lem:F4noint}, \( \eta(f^2_2) \) has a neighbor either in \( \eta(f^2_1 f^2)-\eta(f^2) \) or in \( \eta(f^2_4 f^2)-\eta(f^2) \). We assume the former; the latter is analogous. By Lemma~\ref{lem:F4nooppo} and by $(4,3)$, $(5,1)$, $(5,2)$, $(6,3)$, and $(5,3)$, any neighbor of \( \eta(f^2_4) \) outside the closed domain of \( f^2_4 \) cannot join the domain of \( f^1 \), \( f^1_2 \), \( f^1_3 \), \( f^1_4 \), \( f^2_1 \), or \( f^2_2 \). Hence \( \eta(f^2_4) \) has no neighbor outside the closed domain of \( f^2_4 \), a contradiction.

\smallskip

\noindent\textbf{Case~1.2.} At least one of \( \eta(f^1_1) \) and \( \eta(f^1_4) \) does not join to \( \eta(f^2_3) \).  

\smallskip

If either \( \eta(f^1_1) \) or \( \eta(f^1_4) \) joins to \( \eta(f^2_3) \), then \( G \) contains \( (5,2) \) as a minor. Otherwise, $G$ contains a minor of \( (6,1) \). In either case this is impossible.

\smallskip

\smallskip

\noindent\textbf{Case~2.} \( v \) is adjacent to the neighbor of \( \eta(f^1) \) in \( \eta(f^1 f^1_4) \).

\smallskip

	Denote by \( u \) the neighbor of \( \eta(f^1) \) in \( \eta(f^1 f^1_4) \). 
	
We claim that \( \eta(f^1 f^1_4) \) has no internal vertex and hence \( u = \eta(f^1_4) \), which completes the proof. Suppose otherwise. By the observations above, \( v \) and \( u \) are adjacent, and each has degree four and joins to \( \eta(f^1_3) \). In particular, neither \( v \) nor \( u \) joins to \( \eta(f^1_2) \). By Lemma~\ref{lem:JT}, \( v \) is the neighbor of \( \eta(f^1) \) in \( \eta(f^1 f^1_1) \). It then follows from Lemma~\ref{lem:F4LR} that \( G \) contains a minor of \( (5,3) \), a contradiction.
\end{proof}

\begin{lemma} \label{lem:F4LRE}
	Let \( G \) be a \( 4 \)-connected graph that contains neither \( K_{3,4} \) nor any graph obtained from \( F_4 \) by splitting a vertex as a minor.  
	Suppose \( G \) contains a spanning JT-subdivision \( \eta(F_4) \) of \( F_4 \).  
	For any \( i \in [2] \), \( \eta(f^i_2) \) is adjacent to either \( \eta(f^i_1) \) or \( \eta(f^i_4) \) and has degree four in $G$.
\end{lemma}

\begin{proof}
It follows directly from Lemmas~\ref{lem:JT},~\ref{lem:F4LR},~\ref{lem:F4noint}, and~\ref{lem:F4int} that any neighbor of $\eta(f^i_2)$ outside the closed domain of $f^i_2$ must be $\eta(f^i_1)$ or $\eta(f^i_4)$. By Lemma~\ref{lem:F403} and $(5,3)$, $\eta(f^i_2)$ is adjacent to exactly one of $\eta(f^i_1)$ or $\eta(f^i_4)$, and has degree four.
\end{proof}

\begin{lemma} \label{lem:F4int2}
	Let \( G \) be a \(4\)-connected graph that contains neither \( K_{3,4} \) nor any graph obtained from \( F_4 \) by splitting a vertex as a minor.  
	Suppose \( G \) contains a spanning JT-subdivision \( \eta(F_4) \) of \( F_4 \), and that \( \eta(f^i f^i_2) \), with \( i \in [2] \), contains an internal vertex \( v \).  
	Then \( v \) is adjacent to \( \eta(f^i_3) \) and either \( \eta(f^i_1) \) or \( \eta(f^i_4) \), and has degree four in $G$.
\end{lemma}

\begin{proof}
It follows immediately from Lemmas~\ref{lem:JT},~\ref{lem:F4nooppo},~\ref{lem:F4noint},~\ref{lem:F4int},~\ref{lem:F403}, and $(5,3)$.
\end{proof}

\begin{lemma}[\cite{Lo2025}] \label{lem:F4ind}
	Let \( H \) be a graph obtained from \( F_4 \) by subdividing two independent edges, each with one new vertex, and then joining these two new vertices.  
	Then \( H \) contains \( K_{3,4} \) as a minor.
\end{lemma}

\begin{lemma} \label{lem:F4ind2}
	Let \( H \) be the graph obtained from \( F_4 \) by subdividing \( f^1_1 f^2_4 \) with a new vertex \( v \) and adding four edges \( v f^1_3 \), \( f^1 f^1_3 \), \( f^1_1 f^1_2 \), and \( f^1_4 u \), where \( u \in \{f^1_1, f^1_2, f^2_2, v\} \).  
	Then \( H \) contains \( K_{3,4} \) as a minor.
\end{lemma}

\begin{proof}
	According as \( u \) is \( f^1_1 \), \( f^1_2 \), \( f^2_2 \), or \( v \), the graph \( H \) contains, respectively, \((6,3)\), \((5,3)\), \((4,3)\), or \((8,2)\) as a minor.  
	Hence \( H \) contains a \( K_{3,4} \) minor.
\end{proof}
\begin{lemma} \label{lem:F4ind3}
	Let \( H \) be the graph obtained from \( F_4 \) by subdividing \( f^1_1 f^2_4 \) with a new vertex \( v \) and adding four edges \( v f^1_3 \), \( f^1 f^1_3 \), \( f^1_2 f^1_4 \), and \( f^1_1 u \), where \( u \in \{f^1_2, f^2_1, f^2_2, f^2_3\} \).  
	Then \( H \) contains \( K_{3,4} \) as a minor.
\end{lemma}

\begin{proof}
	According as \( u \) is \( f^1_2 \), \( f^2_1 \), \( f^2_2 \), or \( f^2_3 \), the graph \( H \) contains, respectively, \((5,3)\), \((5,2)\), \((4,3)\), or \((5,1)\) as a minor.  
	Hence \( H \) contains a \( K_{3,4} \) minor.
\end{proof}

\begin{lemma} \label{lem:F4int3}
	Let \( G \) be a \(4\)-connected graph that contains neither \( K_{3,4} \) nor any graph obtained from \( F_4 \) by splitting a vertex as a minor.  
	Suppose \( G \) contains a spanning JT-subdivision \( \eta(F_4) \) of \( F_4 \), and that \( \eta(f^i_1 f^{3-i}_4) \), with \( i \in [2] \), contains an internal vertex \( v \).  
	Then \( v \) is adjacent either to \( \eta(f^i_3) \) and \( \eta(f^i_4) \), or to \( \eta(f^{3-i}_3) \) and \( \eta(f^{3-i}_1) \), and has degree four in $G$.
\end{lemma}

\begin{proof}
	Without loss of generality, assume \( i = 1 \).
		
	By Lemma~\ref{lem:JT}, \( v \) has at least two neighbors outside \( \eta(f^1_1 f^2_4) \).  
	By Lemmas~\ref{lem:F4nooppo},~\ref{lem:F4noint}, and~\ref{lem:F4ind}, and by $(4,3)$, any such neighbor must be  
	\( \eta(f^1_3) \), \( \eta(f^1_4) \), \( \eta(f^2_1) \), or \( \eta(f^2_3) \).
	
	Suppose to the contrary that the lemma fails.  
	Then \( v \) is adjacent to \( \eta(f^1_3) \) and \( \eta(f^2_3) \), to \( \eta(f^1_3) \) and \( \eta(f^2_1) \), to \( \eta(f^1_4) \) and \( \eta(f^2_3) \), or to \( \eta(f^1_4) \) and \( \eta(f^2_1) \).  
	We show that each possibility is impossible.

	\smallskip
	
	\noindent\textbf{Case~1.} \( v \) is adjacent to \( \eta(f^1_3) \) and \( \eta(f^2_3) \).
	
	\smallskip
	
	Recall from Lemma~\ref{lem:F4LRE} that \( \eta(f^1_2) \) is adjacent to exactly one of \( \eta(f^1_1) \) and \( \eta(f^1_4) \), and that \( \eta(f^2_2) \) is adjacent to exactly one of \( \eta(f^2_1) \) and \( \eta(f^2_4) \). Accordingly, depending on the possible choices of these neighbors, we distinguish the following cases.
	
	\smallskip
	
	\noindent\textbf{Case~1.1.} \( \eta(f^1_2) \) joins to \( \eta(f^1_1) \) and \( \eta(f^2_2) \) joins to \( \eta(f^2_4) \).
	
	\smallskip

By Lemmas~\ref{lem:F403},~\ref{lem:F4ind2},~\ref{lem:F4nooppo},~\ref{lem:F4noint}, and~\ref{lem:JT}, \( \eta(f^1_4) \) joins to \( \eta(f^2_3) \) or \( \eta(f^2_4) \), and \( \eta(f^2_1) \) joins to \( \eta(f^1_3) \) or \( \eta(f^1_1) \), which forces \( G \) to contain $(5,1)$, $(5,2)$, or $(6,1)$ as a minor, a contradiction.

\smallskip

\noindent\textbf{Case~1.2.} \( \eta(f^1_2) \) joins to \( \eta(f^1_4) \) and \( \eta(f^2_2) \) joins to \( \eta(f^2_1) \).

\smallskip

By Lemmas~\ref{lem:F403},~\ref{lem:F4ind3},~\ref{lem:F4nooppo},~\ref{lem:F4noint}, and~\ref{lem:JT}, \( \eta(f^1_1) \) joins to \( \eta(f^1_4 f^1) - \eta(f^1) \) and \( \eta(f^2_4) \) joins to \( \eta(f^2_1 f^2) - \eta(f^2) \), which is impossible by $(6,5)$.

\smallskip

\noindent\textbf{Case~1.3.} \( \eta(f^1_2) \) joins to \( \eta(f^1_1) \) and \( \eta(f^2_2) \) joins to \( \eta(f^2_1) \), or \( \eta(f^1_2) \) joins to \( \eta(f^1_4) \) and \( \eta(f^2_2) \) joins to \( \eta(f^2_4) \).

\smallskip

Without loss of generality, we assume the former holds. Then, by Lemmas~\ref{lem:F403},~\ref{lem:F4ind2},~\ref{lem:F4ind3},~\ref{lem:F4nooppo},~\ref{lem:F4noint}, and~\ref{lem:JT}, \( \eta(f^1_4) \) joins to \( \eta(f^2_3) \) or \( \eta(f^2_4) \), and \( \eta(f^2_4) \) joins to \( \eta(f^2_1 f^2) - \eta(f^2) \). Consequently, \( G \) contains a minor of $(6,3)$ or $(7,1)$, a contradiction.

\smallskip

\smallskip

\noindent\textbf{Case~2.} \( v \) is adjacent to \( \eta(f^1_4) \) and \( \eta(f^2_1) \).  

\smallskip

	Then \( G \) contains \((7,6)\) as a minor, a contradiction.

\smallskip

\smallskip

\noindent\textbf{Case~3.} \( v \) is adjacent to \( \eta(f^1_3) \) and \( \eta(f^2_1) \), or to \( \eta(f^1_4) \) and \( \eta(f^2_3) \).  

\smallskip

By symmetry, it suffices to consider the case where \( v \) is adjacent to \( \eta(f^1_3) \) and \( \eta(f^2_1) \). By Lemma~\ref{lem:F4LRE}, \( \eta(f^1_2) \) joins to either \( \eta(f^1_1) \) or \( \eta(f^1_4) \). If it joins to \( \eta(f^1_1) \), then, as in the previous arguments, \( \eta(f^1_4) \) joins to \( \eta(f^2_3) \) or \( \eta(f^2_4) \), forcing \( G \) to contain \((5,2)\) or \( (6,1) \) as a minor. If it joins to \( \eta(f^1_4) \), then, again as before, \( \eta(f^1_1) \) joins to \( \eta(f^1_4 f^1)-\eta(f^1) \), forcing \( G \) to contain \( (6,2) \) as a minor. In either case we obtain a contradiction.
\end{proof}

\begin{lemma} \label{lem:F4int4}
	Let \( G \) be a \(4\)-connected graph that contains neither \( K_{3,4} \) nor any graph obtained from \( F_4 \) by splitting a vertex as a minor.  
	Suppose \( G \) contains a spanning JT-subdivision \( \eta(F_4) \) of \( F_4 \), and that \( \eta(f^i_2 f^{3-i}_3) \), with \( i \in [2] \), contains an internal vertex \( v \).  
	Then \( v \) is adjacent to \( \eta(f^i_3) \) and either \( \eta(f^i_1) \) or \( \eta(f^i_4) \), and has degree four in $G$.
\end{lemma}
\begin{proof}
	It is clear that \( v \) has at least two neighbors outside \( \eta(f^i_2 f^{3-i}_3) \). Moreover, by \((4,1)\), \((4,3)\), Lemma~\ref{lem:F4nooppo}, and Lemma~\ref{lem:F4noint}, any neighbor of \( v \) outside \( \eta(f^i_2 f^{3-i}_3) \) must be one of \( \eta(f^i_1) \), \( \eta(f^i_3) \), or \( \eta(f^i_4) \). By \((5,3)\), \( v \) cannot join to both \( \eta(f^i_1) \) and \( \eta(f^i_4) \). This establishes the lemma.
\end{proof}

\begin{lemma} \label{lem:F41<-24}
	Let \( G \) be a \(4\)-connected graph that contains neither \( K_{3,4} \) nor any graph obtained from \( F_4 \) by splitting a vertex as a minor.  
	Suppose \( G \) contains a spanning JT-subdivision \( \eta(F_4) \) of \( F_4 \), and $\eta(f^i_2)$ is adjacent to $\eta(f^i_j)$ with $i \in [2]$ and $j \in \{1,4\}$. Then \( \eta(f^i_{5-j}) \) is not adjacent to any vertex outside the closed domain of \( f^i_{5-j} \) except for \( \eta(f^i_j) \), \( \eta(f^{3-i}_3) \), and \( \eta(f^{3-i}_{5-j}) \).
\end{lemma}
\begin{proof}	
By Lemmas~\ref{lem:F4noint},~\ref{lem:F4int}, \ref{lem:F4int2}, \ref{lem:F4int3}, and~\ref{lem:F4int4}, any internal vertex adjacent to $\eta(f^i_{5-j})$ must lie in some segment $\eta(e)$ with 
$e \in \{f^i f^i_j, f^i f^i_2, f^i_j f^{3-i}_{5-j}, f^i_2 f^{3-i}_3\}$ 
and must be adjacent to $\eta(f^i_3)$. However, since $\eta(f^i_2)\eta(f^i_j) \in E(G)$, it follows from Lemma~\ref{lem:F403}, $(8,2)$, and $(5,3)$ that no internal vertex is adjacent to $\eta(f^i_{5-j})$.

By Lemmas~\ref{lem:F4nooppo} and~\ref{lem:F4LRE}, and by $(4,3)$, any neighbor of $\eta(f^i_{5-j})$ outside the closed domain of $f^i_{5-j}$ must be $\eta(f^i_j)$, $\eta(f^{3-i}_3)$, or $\eta(f^{3-i}_{5-j})$.
\end{proof}

\begin{proposition} \label{pro:F4}
	Let \( G \) be a \(4\)-connected graph that contains neither \( K_{3,4} \) nor any graph obtained from \( F_4 \) by splitting a vertex as a minor.  
	Suppose \( G \) contains a spanning JT-subdivision \( \eta(F_4) \) of \( F_4 \).
Then one of the following holds:
\begin{itemize}
	\item \( G \) contains \( \mathfrak{D}_{s^1,s^2} \), with \( s^1, s^2 \ge 1 \) and \( s^1+s^2 = |V(G)|-8 \), as a spanning subgraph such that every vertex on either spine has degree four in $G$.
	\item \( G \) contains \( \mathfrak{E}_s \), with \( s = |V(G)|-9 \ge 1 \), as a spanning subgraph such that every vertex on the spine, as well as each of \( \varepsilon^3_1, \varepsilon^3_3, \varepsilon^4 \), has degree four in $G$.
	\item \( G \) contains \( \mathfrak{F} \) as a spanning subgraph such that the vertices \( \varphi^1, \varphi^1_1, \varphi^1_2, \varphi^1_4, \varphi^2, \varphi^2_1, \varphi^2_2, \varphi^2_4 \) have degree four in $G$.
\end{itemize}
\end{proposition}

\begin{proof}
Recall that for $i \in [2]$, $\eta(f^i)$ has degree four and $\eta(f^i)\eta(f^i_3) \in E(G)$ by Lemma~\ref{lem:F403}, and $\eta(f^i_2)$ has degree four and joins to exactly one of $\eta(f^i_1)$ and $\eta(f^i_4)$ by Lemma~\ref{lem:F4LRE}.

So it suffices to consider the cases where $\eta(f^1_2)\eta(f^1_1), \eta(f^2_2)\eta(f^2_1) \in E(G)$ and where $\eta(f^1_2)\eta(f^1_1), \eta(f^2_2)\eta(f^2_4) \in E(G)$.

\smallskip

\noindent\textbf{Case~1.} $\eta(f^1_2)\eta(f^1_1), \eta(f^2_2)\eta(f^2_1) \in E(G)$.

\smallskip

It follows from Lemmas~\ref{lem:F4int}, \ref{lem:F4int2}, \ref{lem:F4int3},~\ref{lem:F4int4}, and~\ref{lem:F41<-24} that, for each \( i \in [2] \), \( \eta(f^i f^i_1) \) contains no internal vertex, and every internal vertex in \( \eta(f^i f^i_2) \), \( \eta(f^i f^i_4) \), \( \eta(f^i_2 f^{3-i}_3) \), or \( \eta(f^i_4 f^{3-i}_1) \) has degree four and joins to \( \eta(f^i_1) \) and \( \eta(f^i_3) \). This determines the neighborhoods of all internal vertices, since every segment not mentioned above contains no internal vertex by Lemma~\ref{lem:F4noint}.

By Lemmas~\ref{lem:F4nooppo},~\ref{lem:F4LRE}, and~\ref{lem:F41<-24}, and by \((4,3)\) and \((6,4)\), it follows that for each \(i\in[2]\) any neighbor of \(\eta(f^i_4)\) outside the closed domain of \(f^i_4\) must be \(\eta(f^i_1)\) or \(\eta(f^{3-i}_3)\). Moreover, by \((6,3)\) one readily deduces that \(\eta(f^1_4)\) and \(\eta(f^2_4)\) have degree four, and that either \(\eta(f^1_4)\eta(f^1_1),\ \eta(f^2_4)\eta(f^2_1)\in E(G)\) or \(\eta(f^1_4)\eta(f^2_3),\ \eta(f^2_4)\eta(f^1_3)\in E(G)\).

\smallskip

\noindent\textbf{Case~1.1.} $\eta(f^1_4)\eta(f^2_3), \eta(f^2_4)\eta(f^2_3) \in E(G)$.

\smallskip

Then there is no internal vertex; otherwise $G$ would contain $(6,3)$ or $(5,2)$ as a minor. Hence $|V(G)| = |V(F_4)|$. In fact, $\mathfrak{F}$ is isomorphic to the subgraph of $G$ obtained from $\eta(F_4)$ by adding the edges $\eta(f^1 f^1_3)$, $\eta(f^1_1 f^1_2)$, $\eta(f^1_4 f^2_3)$, $\eta(f^2 f^2_3)$, $\eta(f^2_1 f^2_2)$, and $\eta(f^2_4 f^1_3)$, where $\varphi^1$, $\varphi^1_1$, $\varphi^1_2$, $\varphi^1_3$, $\varphi^1_4$, $\varphi^2$, $\varphi^2_1$, $\varphi^2_2$, $\varphi^2_3$, $\varphi^2_4$ correspond to $\eta(f^1)$, $\eta(f^1_4)$, $\eta(f^1_2)$, $\eta(f^1_3)$, $\eta(f^1_1)$, $\eta(f^2)$, $\eta(f^2_4)$, $\eta(f^2_2)$, $\eta(f^2_3)$, and $\eta(f^2_1)$, respectively.  

By this correspondence, the vertices $\varphi^1$, $\varphi^1_1$, $\varphi^1_2$, $\varphi^2$, $\varphi^2_1$, and $\varphi^2_2$, corresponding to $\eta(f^1)$, $\eta(f^1_4)$, $\eta(f^1_2)$, $\eta(f^2)$, $\eta(f^2_4)$, and $\eta(f^2_2)$, each have degree four. Furthermore, by $(5,1)$ and $(5,2)$, the vertices $\varphi^1_4$ and $\varphi^2_4$, corresponding to $\eta(f^1_1)$ and $\eta(f^2_1)$, also have degree four.

\smallskip

\noindent\textbf{Case~1.2.} $\eta(f^1_4)\eta(f^1_1), \eta(f^2_4)\eta(f^2_1) \in E(G)$.

\smallskip

It follows that $\mathfrak{D}_{s^1,s^2}$, with $s^1,s^2 \ge 1$ and $s^1+s^2 = |V(G)|-8$, is isomorphic to the subgraph of $G$ obtained from $\eta(F_4)$ by adding the edges $\eta(f^1 f^1_3)$, $\eta(f^1_1 f^1_2)$, $\eta(f^1_1 f^1_4)$, $\eta(f^2 f^2_3)$, $\eta(f^2_1 f^2_2)$, and $\eta(f^2_1 f^2_4)$, together with all edges incident with the internal vertices of $\eta(F_4)$. The vertices $\delta^1_1$, $\delta^1_2$, $\delta^1_3$, $\delta^1_4$, $\delta^2_1$, $\delta^2_2$, $\delta^2_3$, and $\delta^2_4$ correspond to $\eta(f^1_1)$, the neighbor of $\eta(f^2_3)$ in $\eta(f^1_2 f^2_3)$, $\eta(f^1_3)$, the neighbor of $\eta(f^2_1)$ in $\eta(f^1_4 f^2_1)$, the neighbor of $\eta(f^1_1)$ in $\eta(f^2_4 f^1_1)$, $\eta(f^2_3)$, the neighbor of $\eta(f^1_3)$ in $\eta(f^2_2 f^1_3)$, and $\eta(f^2_1)$, respectively.  
Moreover, for each $i \in [2]$, the union of $\eta(f^{3-i}_3 f^i_2)-\eta(f^{3-i}_3)$, $\eta(f^i_2 f^i)$, $\eta(f^i f^i_4)$, and $\eta(f^i_4 f^{3-i}_1)-\eta(f^{3-i}_1)$ forms the spine whose end-vertices are the neighbor of $\eta(f^{3-i}_3)$ in $\eta(f^i_2 f^{3-i}_3)$ and the neighbor of $\eta(f^{3-i}_1)$ in $\eta(f^i_4 f^{3-i}_1)$.

By Lemmas~\ref{lem:F403},~\ref{lem:F4int},~\ref{lem:F4LRE},~\ref{lem:F4int2},~\ref{lem:F4int3}, and~\ref{lem:F4int4}, together with the above discussion, every vertex on either spine has degree four.

\smallskip

\smallskip

\noindent\textbf{Case~2.} $\eta(f^1_2)\eta(f^1_1), \eta(f^2_2)\eta(f^2_4) \in E(G)$.  

\smallskip

By Lemma~\ref{lem:F41<-24}, $(5,1)$, $(5,2)$, $(6,1)$, and $(6,5)$, either $\eta(f^1_4)\eta(f^1_1) \in E(G)$ or $\eta(f^2_1)\eta(f^2_4) \in E(G)$. Without loss of generality, assume $\eta(f^1_4)\eta(f^1_1) \in E(G)$ and $\eta(f^2_1)\eta(f^2_4) \notin E(G)$. So $\eta(f^1_4)$ has degree four. 

By Lemmas~\ref{lem:F4noint},~\ref{lem:F4int}, \ref{lem:F4int2}, \ref{lem:F4int3},~\ref{lem:F4int4}, and~\ref{lem:F41<-24}, and by $(6,5)$, every internal vertex must lie in \(\eta(f^2_3 f^1_2)\), \(\eta(f^1_2 f^1)\), \(\eta(f^1 f^1_4)\), or
\(\eta(f^1_4 f^2_1)\), and be adjacent to
\(\eta(f^1_1)\) and \(\eta(f^1_3)\). 

Furthermore, by Lemma~\ref{lem:F41<-24} and the assumption that $\eta(f^2_1)\eta(f^2_4)\notin E(G)$, the vertex $\eta(f^2_1)$ has no neighbors among the internal vertices and is adjacent to at least one of $\eta(f^1_1)$ and $\eta(f^1_3)$.

\smallskip

\noindent\textbf{Case~2.1.} $\eta(f^2_1) \eta(f^1_1) \in E(G)$.

\smallskip

Then $\mathfrak{E}_s$, where $s = |V(G)| - 9$, is isomorphic to the graph obtained from $\eta(F_4)$ by adding the edges $\eta(f^1)\eta(f^1_3)$, $\eta(f^2)\eta(f^2_3)$, $\eta(f^1_2)\eta(f^1_1)$, $\eta(f^2_2)\eta(f^2_4)$, $\eta(f^1_4)\eta(f^1_1)$, and $\eta(f^2_1)\eta(f^1_1)$, together with all edges incident with the internal vertices of $\eta(F_4)$. More precisely, the vertices $\varepsilon^0$, $\varepsilon^1_1$, $\varepsilon^1_2$, $\varepsilon^1_3$,
$\varepsilon^2$, $\varepsilon^3_1$, $\varepsilon^3_2$, $\varepsilon^3_3$, and $\varepsilon^4$ correspond to the neighbor of $\eta(f^2_1)$ in $\eta(f^2_1 f^1_4)$, $\eta(f^1_1)$, the neighbor of $\eta(f^2_3)$ in $\eta(f^2_3 f^1_2)$, $\eta(f^1_3)$, $\eta(f^2_1)$, $\eta(f^2_4)$, $\eta(f^2_3)$, $\eta(f^2_2)$, and $\eta(f^2)$, respectively, and the union of \(\eta(f^2_3 f^1_2)-\eta(f^2_3)\), \(\eta(f^1_2 f^1)\), \(\eta(f^1 f^1_4)\), and \(\eta(f^1_4 f^2_1)-\eta(f^2_1)\) forms the spine, whose end-vertices are the neighbor of $\eta(f^2_1)$ in $\eta(f^2_1 f^1_4)$ and the neighbor of $\eta(f^2_3)$ in $\eta(f^2_3 f^1_2)$.

By Lemmas~\ref{lem:F403},~\ref{lem:F4int},~\ref{lem:F4LRE}, \ref{lem:F4int2}, \ref{lem:F4int3}, and~\ref{lem:F4int4}, together with the observation that $\eta(f^1_4)$ has degree four, every vertex in the spine has degree four. Moreover, by Lemmas~\ref{lem:F403} and~\ref{lem:F4LRE}, the vertices $\varepsilon^4$ and $\varepsilon^3_3$, corresponding to $\eta(f^2)$ and $\eta(f^2_2)$, also have degree four. Finally, since $\eta(f^2_4)$ is adjacent to neither any vertex on the spine nor to $\eta(f^2_1)$, it follows from $(6,3)$ that $\varepsilon^3_1$, corresponding to $\eta(f^2_4)$, has degree four.

\smallskip

\noindent\textbf{Case~2.2.} $\eta(f^2_1)\eta(f^1_3) \in E(G)$.

\smallskip

Observe that the graph obtained from $\eta(F_4)$ by adding the edges $\eta(f^1)\eta(f^1_3)$, $\eta(f^2)\eta(f^2_3)$, $\eta(f^1_2)\eta(f^1_1)$, $\eta(f^2_2)\eta(f^2_4)$, and $\eta(f^1_4)\eta(f^1_1)$, together with all edges incident to internal vertices, admits an automorphism that swaps $\eta(f^1_1)$ and $\eta(f^1_3)$, swaps $\eta(f^2_2)$ and $\eta(f^2_4)$, and fixes all other vertices. Consequently, we may proceed exactly as in Case~2.1, except that $\eta(f^1_3)$ and $\eta(f^2_2)$ now play the roles of $\eta(f^1_1)$ and $\eta(f^2_4)$, respectively.
\end{proof}

\subsubsection{Proof of Theorem~\ref{thm:4cnppK34mf}}\label{sec:nppolo}

In this section we complete the characterization of $4$-connected non-projective-planar graphs with no $K_{3,4}$ minor by showing that they are precisely the oloidal graphs. As in the previous sections, we assemble in Table~\ref{tab:3} a family of graphs, adopting the same conventions as before.

\begin{table}[ht!]
	\centering
	\caption{A collection of graphs, each containing $K_{3,4}$ as a minor.}
	\label{tab:3}
	\renewcommand{\arraystretch}{2} 
	\begin{tabular}{ccccccc}
		& {\footnotesize 1} & {\footnotesize 2} & {\footnotesize 3} & {\footnotesize 4} & {\footnotesize 5} & {\footnotesize 6} \\
		\noalign{\bigskip}
		{\footnotesize 9} &
		\includegraphics[scale=1, align=c]{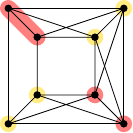} &
		\includegraphics[scale=1, align=c]{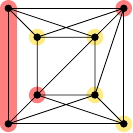} &
		\includegraphics[scale=1, align=c]{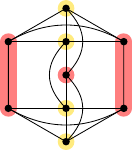} &
		\includegraphics[scale=1, align=c]{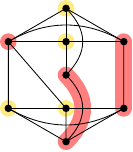} &
		\includegraphics[scale=1, align=c]{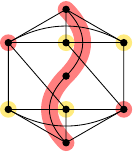} &
		\includegraphics[scale=1, align=c]{b.pdf} \\
	\end{tabular}
\end{table}

\begin{proof}[Proof of Theorem~\ref{thm:4cnppK34mf}]
	The sufficiency follows from Proposition~\ref{pro:olo}. It therefore remains to prove the necessity.
	
	Let \(G\) be a \(4\)-connected non-projective-planar graph that contains no \(K_{3,4}\) minor.

It follows from Propositions~\ref{pro:DEF}, \ref{pro:E20}, and \ref{pro:F4} that every $4$-connected non-projective-planar graph with no $K_{3,4}$ minor falls into one of the six cases listed below. In each case we show that $G$ is oloidal.

\smallskip

\noindent\textbf{Case~1.} $G$ contains \( \mathfrak{D}_{0,0} \) as a spanning subgraph.

\smallskip

If $\delta^1_i \delta^2_j \in E(G) \setminus E(\mathfrak{D}_{0,0})$ for some distinct $i, j \in [4]$, then, by $(9,1)$ and $(9,2)$, no edge in $E(G) \setminus E(\mathfrak{D}_{0,0})$ is incident with $\delta^2_i$ or $\delta^1_j$. Consequently, there exist $i \in [2]$ and $j \in [4]$ such that $E(G) \setminus E(\mathfrak{D}_{0,0}) \subseteq \{ \delta^i_j \delta^{3-i}_k : k \in [4] \setminus \{j\} \}$, or there exist distinct $i, j \in [4]$ such that $E(G) \setminus E(\mathfrak{D}_{0,0}) \subseteq \{ \delta^1_i \delta^2_k, \delta^1_j \delta^2_k : k \in [4] \setminus \{i,j\} \}$.

By the symmetry of $\mathfrak{D}_{0,0}$, we can conclude that $G$ is isomorphic to $\mathfrak{D}_{0,0}+A$, where $A\subseteq \{ \delta^1_1 \delta^2_2,\delta^1_1 \delta^2_3,\delta^1_1 \delta^2_4 \}$ or $A\subseteq \{ \delta^1_1 \delta^2_2, \delta^2_2 \delta^1_3, \delta^1_3 \delta^2_4, \delta^2_4 \delta^1_1 \}$. Thus, $G$ is oloidal.

\smallskip

\smallskip

\noindent\textbf{Case~2.} \( G \) contains \( \mathfrak{D}_{0,1} \) as a spanning subgraph such that each of $\delta^1_2$, $\delta^1_4$, and \( \sigma^2_1 \) has degree four in $G$.

\smallskip

Using the arguments from Case~1, together with the assumption that each of $\delta^1_2$, $\delta^1_4$, and $\sigma^2_1$ has degree four in $G$, one can readily show that there exists $i \in \{1,3\}$ such that $E(G) \setminus E(\mathfrak{D}_{0,1}) \subseteq \{ \delta^1_i \delta^2_j : j \in [4] \setminus \{i\} \}$, or that $E(G) \setminus E(\mathfrak{D}_{0,1}) \subseteq \{ \delta^1_1 \delta^2_2, \delta^2_2 \delta^1_3, \delta^1_3 \delta^2_4, \delta^2_4 \delta^1_1 \}$.

By the symmetry of $\mathfrak{D}_{0,1}$, we can conclude that $G$ is isomorphic to $\mathfrak{D}_{0,1}+A$, where $A\subseteq \{ \delta^1_1 \delta^2_2,\delta^1_1 \delta^2_3,\delta^1_1 \delta^2_4 \}$ or $A\subseteq \{ \delta^1_1 \delta^2_2, \delta^2_2 \delta^1_3, \delta^1_3 \delta^2_4, \delta^2_4 \delta^1_1 \}$. Thus, $G$ is oloidal.

\smallskip

\smallskip

\noindent\textbf{Case~3.} \( G \) contains \( \mathfrak{D}_{s^1,s^2} \), with \( s^1+s^2 \ge 2 \), as a spanning subgraph such that every vertex on either spine has degree four in $G$.

\smallskip

As every vertex on either spine has degree four in $G$, we have $E(G) \setminus E(\mathfrak{D}_{s^1,s^2}) \subseteq \{ \delta^1_1 \delta^2_2, \delta^2_2 \delta^1_3, \delta^1_3 \delta^2_4, \delta^2_4 \delta^1_1 \}$.

We conclude that $G$ is isomorphic to $\mathfrak{D}_{s^1,s^2}+A$, where $A\subseteq \{ \delta^1_1 \delta^2_2, \delta^2_2 \delta^1_3, \delta^1_3 \delta^2_4, \delta^2_4 \delta^1_1 \}$, and hence oloidal.

\smallskip

\smallskip

\noindent\textbf{Case~4.} \( G \) contains \( \mathfrak{E}_0 \) as a spanning subgraph such that each of \( \varepsilon^0 \) and \( \varepsilon^4 \) has degree four in $G$.

\smallskip

It follows from $(9,3)$ and $(9,4)$ that each of $\varepsilon^1_2$ and $\varepsilon^3_1$ has degree four in $G$. Moreover, again by $(9,3)$ and $(9,4)$, if $\varepsilon^2$ is adjacent to one of $\varepsilon^1_3$ and $\varepsilon^3_3$, then the other has degree four in $G$. On the other hand, it follows from $(9,5)$ that $G$ does not contain both edges $\varepsilon^1_1 \varepsilon^3_3$ and $\varepsilon^1_3 \varepsilon^3_2$. Therefore, we conclude that $E(G) \setminus E(\mathfrak{E}_0)$ is a subset of $\{\varepsilon^1_1 \varepsilon^3_2, \varepsilon^1_1 \varepsilon^3_3, \varepsilon^3_3 \varepsilon^2\}$ or $\{\varepsilon^1_1 \varepsilon^3_2, \varepsilon^1_3 \varepsilon^3_2, \varepsilon^1_3 \varepsilon^2\}$.

It is straightforward to verify that $\mathfrak{E}_0+\{\varepsilon^1_1 \varepsilon^3_2, \varepsilon^1_1 \varepsilon^3_3, \varepsilon^3_3 \varepsilon^2\}\cong \mathfrak{E}_0+\{\varepsilon^1_1 \varepsilon^3_2, \varepsilon^1_3 \varepsilon^3_2, \varepsilon^1_3 \varepsilon^2\}$. Hence $G$ is isomorphic to $\mathfrak{E}_0+A$ for some $A\subseteq \{\varepsilon^1_1 \varepsilon^3_2, \varepsilon^1_3 \varepsilon^3_2, \varepsilon^1_3 \varepsilon^2\}$. Thus, $G$ is oloidal.

\smallskip

\smallskip

\noindent\textbf{Case~5.} \( G \) contains \( \mathfrak{E}_s \), with \( s \ge 1 \), as a spanning subgraph such that every vertex on the spine, as well as each of \( \varepsilon^3_1, \varepsilon^3_3, \varepsilon^4 \), has degree four in $G$.

\smallskip

By the assumption that every vertex on the spine, as well as each of $\varepsilon^3_1$, $\varepsilon^3_3$, and $\varepsilon^4$, has degree four in $G$, we conclude that
$E(G) \setminus E(\mathfrak{E}_s) \subseteq \{\varepsilon^1_1 \varepsilon^3_2, \varepsilon^1_3 \varepsilon^3_2, \varepsilon^1_3 \varepsilon^2\}$.

Therefore, $G$ is isomorphic to $\mathfrak{E}_s+A$ for some $A\subseteq \{\varepsilon^1_1 \varepsilon^3_2, \varepsilon^1_3 \varepsilon^3_2, \varepsilon^1_3 \varepsilon^2\}$, and is oloidal.

\smallskip

\smallskip

\noindent\textbf{Case~6.} \( G \) contains \( \mathfrak{F} \) as a spanning subgraph such that the vertices \( \varphi^1, \varphi^1_1, \varphi^1_2, \varphi^1_4, \varphi^2, \varphi^2_1, \varphi^2_2, \varphi^2_4 \) have degree four in $G$.

\smallskip

It is immediate that \(E(G)\setminus E(\mathfrak{F}) \subseteq \{\varphi^1_3 \varphi^2_3\}\). Hence, \(G\) is isomorphic to \(\mathfrak{F}+A\) for some \(A\subseteq \{\varphi^1_3 \varphi^2_3\}\), and therefore \(G\) is oloidal.
\end{proof}

\subsection{Proof of Theorem~\ref{thm:4cK34mf}}\label{sec:4cK34mf}

We now proceed to merge the results of the projective-planar and non-projective-planar cases, thereby establishing the proof of our principal characterization theorem.

\begin{proof}[Proof of Theorem~\ref{thm:4cK34mf}]
By Theorems~\ref{thm:MS} and~\ref{thm:4cnppK34mf}, a graph $G$ is $4$-connected and has no $K_{3,4}$ minor if and only if, depending on whether $G$ is planar, non-planar but projective-planar, or non-projective-planar, it is respectively a $4$-connected planar graph, a $4$-connected non-planar subgraph of a patch graph or isomorphic to $K_6$, or an oloidal graph. It remains to note that the class of $4$-connected subgraphs of patch graphs coincides with the class of $4$-connected subgraphs of reduced patch graphs. This completes the proof.
\end{proof}

\section{Generating all graphs with no $K_{3,4}$ minor}\label{sec:K34}

The $4$-connected graphs with no $K_{3,4}$ minor have been characterized in Theorem~\ref{thm:4cK34mf}. In this section, we show that general graphs with no $K_{3,4}$ minor can be constructed from the $4$-connected ones.

We first show how $3$-connected graphs with no $K_{3,4}$ minor can be reduced to the $4$-connected case. Then all graphs with no $K_{3,4}$ minor can be generated from the $3$-connected ones via the standard clique sum construction.

We begin with the following observation.

\begin{lemma}\label{lem:3conK34mf}
	Let $G$ be a $3$-connected graph with a $3$-cut $S$, and let $v_1, v_2 \in S$ be two non-adjacent vertices. If $G$ has no $K_{3,4}$ minor, then $G + v_1v_2$ also has no $K_{3,4}$ minor.
\end{lemma}

\begin{proof}
	Suppose, to the contrary, that $G + v_1v_2$ has a $K_{3,4}$ minor. Let $\mu$ be a spanning model of $K_{3,4}$ in $G + v_1v_2$.
	
	Denote the vertex set of $K_{3,4}$ by $X \cup Y$, where $|X| = 3$ and $|Y| = 4$, so that two vertices are adjacent if and only if they belong to different partite sets.
	
	Since $G$ has no $K_{3,4}$ minor, there exists $w \in V(K_{3,4})$ such that $\{v_1,v_2\} \subseteq \mu(w)$, or, without loss of generality, there exist vertices $x \in X$ and $y \in Y$ such that $v_1 \in \mu(x)$ and $v_2 \in \mu(y)$.
	
	\smallskip
	\noindent\textbf{Case~1.} There exists $w \in V(K_{3,4})$ such that $\{v_1,v_2\} \subseteq \mu(w)$.
	
	\smallskip
	
	We claim that there exists a component $A$ of $G - S$ such that, for every vertex $v \in V(K_{3,4})$ with $\mu(v) \cap S = \emptyset$, we have $\mu(v) \subseteq V(A)$. Suppose otherwise. Then there exist two distinct components $A$ and $A'$ of $G - S$ and vertices $v, v' \in V(K_{3,4})$ such that $\mu(v) \subseteq V(A)$ and $\mu(v') \subseteq V(A')$, with neither intersecting $S$. Consequently, $v$ and $v'$ belong to the same partite set, either both in $X$ or both in $Y$. This is impossible, since $A$ and $A'$ are separated by $S$, while $\{v_1,v_2\} \subseteq \mu(w)$ for some $w \in V(K_{3,4})$. This proves the claim.
	
	We now define a spanning model $\mu'$ of $K_{3,4}$ in $G$. Choose an arbitrary component $\tilde{A}$ of $G - S$ distinct from $A$. For every vertex $v \in V(K_{3,4})$ with $\mu(v) \subseteq V(A)$, set $\mu'(v) := \mu(v)$. Let $z$ be the vertex of $K_{3,4}$ such that $\mu(z)$ contains the third vertex of $S$. Set $\mu'(w) := \mu(w) \cup V(\tilde{A})$, and, if $z \neq w$, set $\mu'(z) := \mu(z) \setminus V(\tilde{A})$. This yields a spanning model of $K_{3,4}$ in $G$, contradicting our assumption.
	
	\smallskip
	\noindent\textbf{Case~2.} There exist vertices $x \in X$ and $y \in Y$ such that $v_1 \in \mu(x)$ and $v_2 \in \mu(y)$.
	
	\smallskip
	
	We proceed similarly to Case~1. By the same argument, there exists a component $A$ of $G - S$ such that $\mu(v) \subseteq V(A)$ for every vertex $v \in V(K_{3,4})$ with $\mu(v) \cap S = \emptyset$.
	
We again derive a contradiction by defining a spanning model $\mu'$ of $K_{3,4}$ in $G$. Choose an arbitrary component $\tilde{A}$ of $G - S$ with $\tilde{A} \neq A$. For every vertex $v \in V(K_{3,4})$ with $\mu(v) \subseteq V(A)$, set $\mu'(v) := \mu(v)$. Let $z \in V(K_{3,4})$ be such that $\mu(z)$ contains the third vertex of $S$. If $x \neq z$ and $z \in X$, set $\mu'(x) := \mu(x) \setminus V(\tilde{A})$, $\mu'(y) := \mu(y) \cup V(\tilde{A})$, and $\mu'(z) := \mu(z) \setminus V(\tilde{A})$. If $x = z$, set $\mu'(x) := \mu(x) \cup V(\tilde{A})$ and $\mu'(y) := \mu(y) \setminus V(\tilde{A})$. The cases $y \neq z$ with $z \in Y$ and $y = z$ are treated analogously.
\end{proof}

A \emph{$k$-clique} of $G$ is a subgraph of $G$ on $k$ vertices in which every pair of vertices is adjacent. We say that a $3$-connected graph $G$ is \emph{clamped} if every $3$-cut of $G$ induces a $3$-clique. Lemma~\ref{lem:3conK34mf} immediately yields the following corollary.

\begin{corollary}\label{cor:3conK34mf}
	Every $3$-connected graph with no $K_{3,4}$ minor is a spanning subgraph of a clamped $3$-connected graph with no $K_{3,4}$ minor.
\end{corollary}

We follow the approach of~\cite{Maharry2000}, with a slight modification, to decompose a clamped $3$-connected graph into graphs that are either $4$-connected or isomorphic to $K_4$. We only give an outline of the framework and refer to~\cite{Maharry2000} for further details.

The decomposition is based on the following observation. Let $G$ be a clamped $3$-connected graph. If $G$ is neither $4$-connected nor isomorphic to $K_4$, then $G$ has a $3$-cut $S$, and there exist subgraphs $G_1$ and $G_2$ of $G$ such that $V(G_1)\cap V(G_2)=S$, $E(G_1)\cup E(G_2)=E(G)$, and $E(G_1)\cap E(G_2)=E(G[S])$. Since any $3$-cut of $G_1$ or $G_2$ is also a $3$-cut of $G$, both $G_1$ and $G_2$ are clamped $3$-connected. Therefore, any clamped $3$-connected graph that is neither $4$-connected nor isomorphic to $K_4$ can be decomposed into two smaller clamped $3$-connected graphs.

This motivates the following structure. Let $\mathcal{G}$ be the family of subgraphs of $G$ that are either edge-maximal $4$-connected, or isomorphic to $K_4$ and not contained in any edge-maximal $4$-connected subgraph of $G$. Let $\mathcal{S}$ be the family of $3$-cliques in $G$ that are induced by $3$-cuts.

Define $\mathcal{T}$ to be the graph with vertex set $\mathcal{G}\cup\mathcal{S}$ such that every edge has one end-vertex in $\mathcal{G}$ and the other in $\mathcal{S}$, and where $H\in\mathcal{G}$ is adjacent to $\Delta\in\mathcal{S}$ if and only if $H$ contains $\Delta$ as a subgraph of $G$. Then $\mathcal{T}$ is a tree, and we call $(\mathcal{T}, \mathcal{G}, \mathcal{S})$ the \emph{clamped tree decomposition} of $G$. The graph $G$ can be reconstructed by first taking the disjoint union of the graphs in $\mathcal{G}$ and then, for each $\Delta\in\mathcal{S}$, identifying the copies of $\Delta$ in those graphs from $\mathcal{G}$ that are adjacent to $\Delta$ in $\mathcal{T}$. For $H \in \mathcal{G}$, define the \emph{closure} of $H$, denoted by $\widehat{H}$, to be the graph obtained from $H$ by adding, for each $\Delta \in \mathcal{S}$ that is adjacent to $H$ in $\mathcal{T}$, a new vertex $v^H_{\Delta}$ adjacent to the three vertices of $\Delta$. Define $\widehat{\mathcal{G}} := \{ \widehat{H} : H \in \mathcal{G} \}$.

Since every cycle of length three in a $4$-connected planar graph is a facial cycle, it follows that for any $H \in \mathcal{G}$, the graph $H$ is planar if and only if its closure $\widehat{H}$ is planar.

We need the following two lemmas. The first is an immediate consequence of~\cite[(2.4)]{Robertson1990}, and the second follows from~\cite[Proposition~3.2]{Maharry2012}.

\begin{lemma}[\cite{Robertson1990}]\label{lem:rural}
	Let $G$ be a $3$-connected non-planar graph with a vertex $v$ of degree three. Then $G$ contains a subdivision of $K_{3,3}$ that contains $v$ and in which $v$ has degree three.
\end{lemma}

\begin{lemma}[\cite{Maharry2012}]\label{lem:MS}
	Let $H_1$ be a graph containing a $3$-clique $\Delta_1$, and let $H_2$ be a graph containing a $3$-clique $\Delta_2$. Suppose that the graph obtained from $H_1$ by adding a new vertex adjacent to all vertices of $\Delta_1$ is planar, and that the graph obtained from $H_2$ by adding a new vertex adjacent to all vertices of $\Delta_2$ has no $K_{3,4}$ minor. Then the graph obtained from the disjoint union of $H_1$ and $H_2$ by identifying $\Delta_1$ with $\Delta_2$ contains no $K_{3,4}$ minor.
\end{lemma}

The following proposition characterizes the clamped $3$-connected graphs with no $K_{3,4}$ minor in terms of their clamped tree decompositions.

\begin{proposition}\label{pro:c3conK34mf}
	Let $G$ be a clamped $3$-connected graph, and let $(\mathcal{T}, \mathcal{G}, \mathcal{S})$ be its clamped tree decomposition. Then $G$ contains no $K_{3,4}$ minor if and only if every closure $\widehat{H}\in\widehat{\mathcal{G}}$ contains no $K_{3,4}$ minor and one of the following holds:
	\begin{itemize}
		\item Exactly one vertex of $\mathcal{S}$ has degree three in $\mathcal{T}$, all other vertices of $\mathcal{S}$ have degree two in $\mathcal{T}$, and every graph in ${\mathcal{G}}$ is planar.
		\item Every vertex of $\mathcal{S}$ has degree two in $\mathcal{T}$, and at most one graph in ${\mathcal{G}}$ is non-planar.
	\end{itemize}
\end{proposition}

\begin{proof}
	We first prove the necessity.
	
For every $H \in \mathcal{G}$, we have $\widehat{H} \preceq G$, since it is a subgraph of the graph obtained from $G$ by contracting every component of $G - V(H)$. Hence, $\widehat{H}$ contains no $K_{3,4}$ minor.
	
	If some $\Delta \in \mathcal{S}$ has degree at least four in $\mathcal{T}$, then the union $U$ of the neighbors of $\Delta$ in $\mathcal{T}$ forms a subgraph of $G$ that contains a $K_{3,4}$ minor. More precisely, one obtains $K_{3,4}$ by contracting each component of $U - V(\Delta)$ and deleting the edges of $\Delta$. Therefore, every vertex of $\mathcal{S}$ has degree two or three in $\mathcal{T}$.
	
	If there are distinct vertices $\Delta_1, \Delta_2 \in \mathcal{S}$ of degree three in $\mathcal{T}$, then the union of the neighbors of $\Delta_1$ and $\Delta_2$ together with the elements of $\mathcal{G}$ lying on the path of $\mathcal{T}$ joining $\Delta_1$ and $\Delta_2$ contains a $K_{3,4}$ minor (as there are three disjoint paths between $\Delta_1$ and $\Delta_2$ in $G$). Hence, there is at most one vertex of $\mathcal{S}$ of degree three in $\mathcal{T}$.
	
	\smallskip
	\noindent\textbf{Case~1.} Exactly one vertex of $\mathcal{S}$ has degree three in $\mathcal{T}$.
	
	\smallskip
	
	Let $\Delta_1 \in \mathcal{S}$ be this vertex. Suppose, to the contrary, that there exists a non-planar graph ${H} \in {\mathcal{G}}$. Let $\Delta_2$ be the neighbor of $H$ on the path in $\mathcal{T}$ joining $\Delta_1$ and $H$.
	
Let $\widetilde{H}$ be the subgraph of $G$ such that $H$ is a subgraph of $\widetilde{H}$, $\widetilde{H} - V(\Delta_2)$ is connected, and, subject to these conditions, $|E(\widetilde{H})|$ is maximized. Since ${H}$ is non-planar, the graph obtained from $\widetilde{H}$ by adding a vertex $v$ adjacent to the vertices of $\Delta_2$ is also non-planar and hence, by Lemma~\ref{lem:rural}, contains a subdivision $\eta(K_{3,3})$ of $K_{3,3}$ that contains $v$ and in which $v$ has degree three.

	Since there are three disjoint paths in $G$ joining $\Delta_1$ and $\Delta_2$, it follows that the union of these three paths, the elements of $\mathcal{G}$ containing $\Delta_1$, and $\eta(K_{3,3}) - v$ contains a $K_{3,4}$ minor, a contradiction. Therefore, every graph in ${\mathcal{G}}$ is planar.

	\smallskip
	
	\smallskip
	\noindent\textbf{Case~2.} No vertex of $\mathcal{S}$ has degree three in $\mathcal{T}$.
	
	\smallskip
	
Suppose, to the contrary, that there exist two non-planar graphs ${H_1}, {H_2} \in {\mathcal{G}}$. For $i \in [2]$, let $\Delta_i$ be the neighbor of $H_i$ on the path in $\mathcal{T}$ joining $H_1$ and $H_2$, and let $\widetilde{H_i}$ be the subgraph of $G$ such that $H_i$ is a subgraph of $\widetilde{H_i}$, $\widetilde{H_i} - V(\Delta_i)$ is connected, and, subject to these conditions, $|E(\widetilde{H_i})|$ is maximized. 

As before, Lemma~\ref{lem:rural} implies that the graph obtained from $\widetilde{H_i}$ by adding a vertex $v_i$ adjacent to the vertices of $\Delta_i$ contains a subdivision $\eta_i(K_{3,3})$ of $K_{3,3}$ that contains $v_i$ and in which $v_i$ has degree three. 

As there exist three disjoint paths in $G$ joining $\Delta_1$ and $\Delta_2$, it follows that the union of $\eta_1(K_{3,3}) - v_1$, $\eta_2(K_{3,3}) - v_2$, and these three paths contains a $K_{3,4}$ minor, a contradiction. Hence, at most one graph in ${\mathcal{G}}$ is non-planar.

	\smallskip
	
	\smallskip
	
	We now prove the sufficiency.
	
	By assumption, for each $H \in \mathcal{G}$, the closure $\widehat{H}$ contains no $K_{3,4}$ minor. Moreover, $\widehat{H}$ is planar if and only if $H$ is planar.
	
Recall that $G$ can be constructed by identifying $3$-cliques in $\mathcal{S}$ among the graphs $H \in \mathcal{G}$.

Equivalently, for each $3$-clique $\Delta\in\mathcal{S}$, one identifies the copies of $\Delta$ in the closures $\widehat{H}\in\widehat{\mathcal{G}}$ that contain $\Delta$, and then removes from each such $\widehat{H}$ the vertex adjacent to $\Delta$ that was added in forming the closure (denoted $v^H_\Delta$ in the definition). After all identifications are performed and the added vertices are removed, the closures are merged into the graph $G$. This reformulation is convenient for applying Lemma~\ref{lem:MS}.

It remains to show that, in each of the following two cases, the constructed graph $G$ has no $K_{3,4}$ minor.

	\smallskip
	\noindent\textbf{Case~1.} At most one graph in ${\mathcal{G}}$ is non-planar and every vertex of $\mathcal{S}$ has degree two in $\mathcal{T}$.
	
	\smallskip
	
Let $H_1 \in \mathcal{G}$ be such that every other graph in $\mathcal{G}$ has a planar closure. Order the graphs in $\mathcal{G}$ as $H_1, H_2, \dots, H_t$, where $t = |\mathcal{G}|$, so that for each $i \in [t]$ the union of $H_1, \dots, H_i$ forms a clamped $3$-connected subgraph of $G$. Starting from $\widehat{H_1}$, we successively merge $\widehat{H_2}, \dots, \widehat{H_t}$ in this order via $3$-clique identifications, deleting two degree three vertices at each step (since each vertex of $\mathcal{S}$ has degree two in $\mathcal{T}$). As $\widehat{H_1}$ contains no $K_{3,4}$ minor and $\widehat{H_2}, \dots, \widehat{H_t}$ are planar, Lemma~\ref{lem:MS} implies that the resulting graph $G$ contains no $K_{3,4}$ minor.
	
	\smallskip
	
	\smallskip
	\noindent\textbf{Case~2.} Every graph in ${\mathcal{G}}$ is planar, exactly one vertex $\Delta \in \mathcal{S}$ has degree three in $\mathcal{T}$, and all other vertices of $\mathcal{S}$ have degree two in $\mathcal{T}$.
	
	\smallskip
	
Let $H_1$ and $H_2$ be two neighbors of $\Delta$ in $\mathcal{T}$, and let $v_1$ and $v_2$ be the degree three vertices of $\widehat{H_1}$ and $\widehat{H_2}$, respectively, that were added in forming the closures and are adjacent to the vertices of $\Delta$.

	Let $\widetilde{H}$ be obtained from the disjoint union of $\widehat{H_1}$ and $\widehat{H_2}$ by identifying the copies of $\Delta$ and deleting $v_2$. 
	
	We claim that $\widetilde{H}$ contains no $K_{3,4}$ minor. Suppose otherwise, and let $\mu$ be a spanning model of $K_{3,4}$ in $\widetilde{H}$. Denote the vertex set of $K_{3,4}$ by $X \cup Y$, where $|X| = 3$ and $|Y| = 4$, and where two vertices are adjacent if and only if they belong to different partite sets $X$ and $Y$.

Since $\widetilde{H} - v_1$ is planar, we have $v_1 \in \mu(x)$ for some $x \in X$. As $v_1$ has degree three in $\widetilde{H}$ while $x$ has degree four in $K_{3,4}$, at least one neighbor of $v_1$ lies in $\mu(x)$. Let $u$ be such a neighbor. As every neighbor of $v_1$ other than $u$ is also a neighbor of $u$, it follows immediately that $\widetilde{H} - v_1$ has a $K_{3,4}$ minor, which is impossible.

	Now, as in Case~1, we may successively merge $\widetilde{H}$ with the graphs in $\widehat{\mathcal{G}}$ other than $\widehat{H_1}$ and $\widehat{H_2}$ to obtain $G$. By Lemma~\ref{lem:MS}, the resulting graph contains no $K_{3,4}$ minor, since $\widetilde{H}$ has no $K_{3,4}$ minor and all graphs in $\widehat{\mathcal{G}}$ are planar.
\end{proof}

Let $H_1, H_2$ be two graphs, each containing a $k$-clique. By identifying these $k$-cliques in the disjoint union of $H_1$ and $H_2$, we obtain a \emph{$k$-clique sum} of $H_1$ and $H_2$. The following proposition is straightforward (cf.~\cite{Maharry2012}); we omit the proof.

\begin{proposition}\label{pro:conK34mfcs}
	Let $H_1,H_2$ be graphs each containing a $k$-clique, where $k \in [2]$. Then $H_1$ and $H_2$ contain no $K_{3,4}$ minor if and only if any $k$-clique sum of $H_1$ and $H_2$ contains no $K_{3,4}$ minor.
\end{proposition}

Propositions~\ref{pro:conK34mfcs} and~\ref{pro:c3conK34mf} show how a connected graph without a $K_{3,4}$ minor can be reduced to $4$-connected graphs without a $K_{3,4}$ minor. Together with Theorem~\ref{thm:4cK34mf}, this yields a characterization of graphs with no $K_{3,4}$ minor.

\section{Applications}\label{sec:applications}

In this section we derive three consequences of Theorem~\ref{thm:4cK34mf} concerning edge density conditions forcing the presence of certain minors, hamiltonian-connectedness, and embeddability on the torus. In particular, we establish Theorems~\ref{thm:K34K6-}, \ref{thm:hc}, and \ref{thm:T} in Sections~\ref{sec:K34K6-}, \ref{sec:hc}, and \ref{sec:T}, respectively.

\subsection{Edge density and $K_{3,4}$ minors}\label{sec:K34K6-}

A classical problem in graph minor theory is to determine, for a graph class $\mathcal{G}$ and a graph $H$, a lower bound on the average degree (respectively, the minimum degree) that forces every graph in $\mathcal{G}$ to contain $H$ as a minor.

The following theorem is due to J{\o}rgensen~\cite{Jorgensen2001}, and concerns $K_{3,4}$ minors in the class of $3$-connected graphs.

\begin{theorem}[\cite{Jorgensen2001}]\label{thm:Jorgensen}
	Every $3$-connected graph $G$ with at least $3|V(G)|-3$ edges is either isomorphic to $K_6$ or contains $K_{3,4}$ as a minor. Moreover, every graph with minimum degree at least six contains $K_{3,4}$ as a minor.
\end{theorem}

Both statements in Theorem~\ref{thm:Jorgensen} are tight. As mentioned (without explicit construction) in~\cite{Jorgensen2001}, there exist infinitely many $3$-connected graphs $G$ with $3|V(G)|-4$ edges that contain no $K_{3,4}$ minor. Note also that there are planar graphs with minimum degree exactly five.

We complement Theorem~\ref{thm:Jorgensen} with the following result.

\begin{corollary}\label{cor:Jorgensen}
	There exist infinitely many $4$-connected graphs $G$ with $3|V(G)|-4$ edges that contain no $K_{3,4}$ minor. Moreover, every $4$-connected non-planar graph with minimum degree at least five is either isomorphic to $K_6$ or contains $K_{3,4}$ as a minor.
\end{corollary}

\begin{proof}
	For the first statement, consider $G := \mathfrak{D}_{s^1,s^2} + \{\delta^1_1 \delta^2_2, \delta^2_2 \delta^1_3, \delta^1_3 \delta^2_4, \delta^2_4 \delta^1_1\}$, with $s^1, s^2 \ge 0$. Then $G$ has no $K_{3,4}$ minor and has $3|V(G)| - 4$ edges.
	
	For the second statement, we apply Theorem~\ref{thm:4cK34mf}. It suffices to observe that both patch graphs and oloidal graphs are $4$-degenerate. The $4$-degeneracy of patch graphs follows from Lemma~\ref{lem:pg4d}, and the same property for oloidal graphs is easily verified.
\end{proof}

We now proceed to the proof of Theorem~\ref{thm:K34K6-}, using the following theorem of Mader~\cite{Mader1968}.

\begin{theorem}[\cite{Mader1968}]\label{thm:Mader}
	Every graph with minimum degree at least five contains $K_6^-$ or the icosahedron as a minor.
\end{theorem}

\begin{proof}[Proof of Theorem~\ref{thm:K34K6-}]
	By Theorem~\ref{thm:Mader} and Corollary~\ref{cor:Jorgensen}, it suffices to show that every $4$-connected non-planar graph $G$ containing the icosahedron as a minor also contains $K_6^-$ as a minor. Denote by $H$ the icosahedron. Then $G$ and $H$ are $4$-connected, with $G$ non-planar and $H$ planar. This permits the application of Theorem~\ref{thm:HT}.
	
	Theorem~\ref{thm:HT} lists seven possible extensions. Since all facial cycles of $H$ have length three, only three of these need consideration. Hence $G$ contains a minor $H'$ obtained from $H$ in one of the following ways:
	
	\begin{itemize}
		\item $H'$ arises from joining two non-cofacial vertices of $H$. By symmetry, $H'$ is isomorphic to one of the two graphs depicted in Figure~\ref{fig:ikosa12}.
		
		\item $H'$ arises from performing a non-planar split on $H$. In this case, $H'$ is isomorphic to the graph in Figure~\ref{fig:ikosa3}.
		
		\item There exist facial cycles $C_1$ and $C_2$ of $H$ sharing an edge $uv$, and $H'$ is obtained by splitting $u$ along $C_1$ into $u_1,u_2$ and $v$ along $C_2$ into $v_1,v_2$, with $u_1$ adjacent to $v_1$ in the intermediate graph, followed by joining $u_2$ to $v_2$. Then $H'$ contains the graph of Figure~\ref{fig:ikosa4} as a spanning subgraph.
	\end{itemize}
	In each of the three cases, $H'$ contains $K_6^-$ as a minor. This concludes the proof.
\end{proof}

\begin{figure}[!ht]
	\centering
	\includegraphics[scale=1]{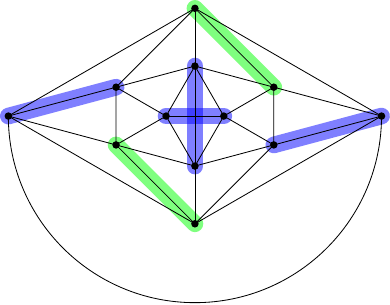}
	\hspace{30pt}
	\includegraphics[scale=1]{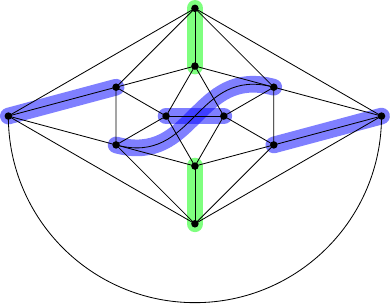}
	\caption{The two graphs obtained from the icosahedron by joining two non-cofacial vertices.}
	\label{fig:ikosa12}
\end{figure}

\begin{figure}[!ht]
	\centering
	\includegraphics[scale=1]{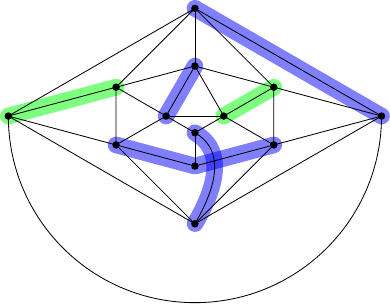}
	\caption{The graph obtained from the icosahedron by a non-planar split.}
	\label{fig:ikosa3}
\end{figure}

\begin{figure}[!ht]
	\centering
	\includegraphics[scale=1]{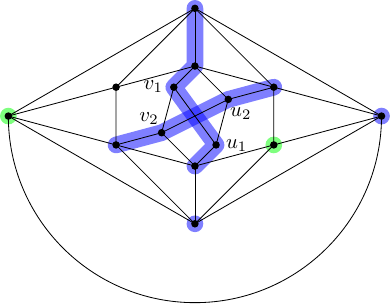}
	\caption{A common spanning subgraph of graphs obtained from the icosahedron by planar splits along facial cycles $C_1$ and $C_2$ sharing an edge $uv$, with $u$ split into $u_1,u_2$ and $v$ split into $v_1,v_2$, $u_1$ adjacent to $v_1$, and $u_2$ joined to $v_2$.}
	\label{fig:ikosa4}
\end{figure}

\subsection{Hamiltonian-connectedness}\label{sec:hc}

The following theorem is due to Kawarabayashi and Ozeki~\cite{Kawarabayashi2015}. Together with Theorem~\ref{thm:4cK34mf}, it yields a direct proof of Theorem~\ref{thm:hc}, for which we provide only a brief sketch.

\begin{theorem}[\cite{Kawarabayashi2015}]\label{thm:KO}
	Every $4$-connected projective-planar graph is hamiltonian-connected.
\end{theorem}

\begin{proof}[Sketch of the proof of Theorem~\ref{thm:hc}]
Patch graphs and $K_6$ are projective-planar. Hence, by Theorems~\ref{thm:KO} and~\ref{thm:4cK34mf}, it suffices to show that the edge-minimal oloidal graphs $\mathfrak{D}_{s^1,s^2}$, $\mathfrak{E}_s$, and $\mathfrak{F}$ are hamiltonian-connected. The verification for these cases is routine.
\end{proof}

\subsection{Toroidal embeddings without $K_{3,4}$ minors}\label{sec:T}

In what follows, we show that Theorem~\ref{thm:T} is a consequence of Theorem~\ref{thm:4cK34mf}.

\begin{proof}[Proof of Theorem~\ref{thm:T}]
	We apply Theorem~\ref{thm:4cK34mf}. 
	
	It is well known that $K_6$ embeds on the torus.
	
Next, every non-planar patch graph admits a toroidal embedding. Indeed, each such graph can be obtained from a planar graph bounded by a facial cycle of length four (corresponding to the facial walk of length four in the initial patch graph) by identifying each pair of non-consecutive vertices on that cycle. It then follows immediately that the resulting graph can be embedded on the torus.

	It remains to show that every oloidal graph embeds on the torus. Such embeddings are shown in Figure~\ref{fig:T} for the edge-maximal cases, namely $\mathfrak{D}_{s^1,s^2} + \{\delta^1_1 \delta^2_2, \delta^1_1 \delta^2_3, \delta^1_1 \delta^2_4\}$, $\mathfrak{D}_{s^1,s^2} + \{\delta^1_1 \delta^2_2, \delta^2_2 \delta^1_3, \delta^1_3 \delta^2_4, \delta^2_4 \delta^1_1\}$, $\mathfrak{E}_s + \{\varepsilon^1_1 \varepsilon^3_2, \varepsilon^1_3 \varepsilon^3_2, \varepsilon^1_3 \varepsilon^2\}$, and $\mathfrak{F} + \{\varphi^1_3 \varphi^2_3\}$.
	The embedding of $\mathfrak{E}_s + \{\varepsilon^1_1 \varepsilon^3_2, \varepsilon^1_3 \varepsilon^3_2, \varepsilon^1_3 \varepsilon^2\}$ is omitted, since $\mathfrak{E}_s + \{\varepsilon^1_1 \varepsilon^3_2, \varepsilon^1_3 \varepsilon^3_2, \varepsilon^1_3 \varepsilon^2\} \cong \mathfrak{D}_{0,s+1} + \{\delta^1_1 \delta^2_2, \delta^1_1 \delta^2_3, \delta^1_1 \delta^2_4\}$.
	
	This completes the proof.
\end{proof}

\begin{figure}[!ht]
	\centering
	\includegraphics[scale=1]{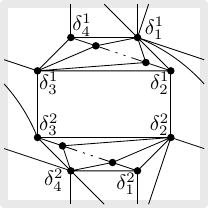}
	\hspace{30pt}
	\includegraphics[scale=1]{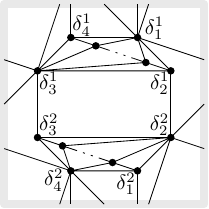}
	\hspace{30pt}
	\includegraphics[scale=1]{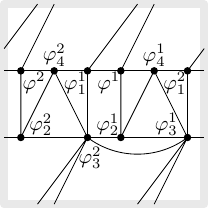}
\caption{Toroidal embeddings of $\mathfrak{D}_{s^1,s^2} + \{\delta^1_1 \delta^2_2, \delta^1_1 \delta^2_3, \delta^1_1 \delta^2_4\}$ (left), $\mathfrak{D}_{s^1,s^2} + \{\delta^1_1 \delta^2_2, \delta^2_2 \delta^1_3, \delta^1_3 \delta^2_4, \delta^2_4 \delta^1_1\}$ (middle), and $\mathfrak{F} + \{\varphi^1_3 \varphi^2_3\}$ (right), where opposite sides of the bounding squares are identified to form the torus.}
	\label{fig:T}
\end{figure}




\bibliographystyle{abbrv}
\bibliography{paper}

\end{document}